\documentclass[
amsmath,amssymb,
]{revtex4-2}
\usepackage{hyperref}
\usepackage{ulem}
\usepackage{xcolor}
\usepackage{graphicx}
\usepackage{dcolumn}
\usepackage{bm}
\usepackage{amsthm}
\newcounter{mainEq}

\usepackage{appendix}

\newtheorem{theorem}{Theorem}
\newtheorem{lemma}{Lemma}
\newtheorem{corollary}{Corollary}

\begin{document}

\title{Watanabe--Strogatz Invariants in the Liouvillian Dynamics of Coupled Phase Oscillators via the Koopman Framework}

\author{Keisuke Taga$^{1}$, and Hiroya Nakao$^{2}$}

\address{$^{1}$Department of Physics and Astronomy, Tokyo University of Science, Chiba 278-8510, Japan\\
$^{2}$Department of Systems and Control Engineering and Research Center for Autonomous Systems Materialogy, Institute of Science Tokyo, Tokyo 152-8552, Japan}

\email{tagakeisuke@rs.tus.ac.jp}

\begin{abstract}
In dynamical systems, invariants, i.e., constants of motion conserved along the trajectory, play important roles in characterizing the system's dynamical behavior. Recent applications of the Koopman operator framework to nonlinear dynamical systems have provided new insights into the invariants. For a certain class of globally coupled phase oscillators, which serve as models for various synchronization phenomena, Watanabe and Strogatz proved the existence of $N-3$ invariants in $N$ oscillator systems.
In this study, we derive these invariants
from an operator-theoretic perspective by exploiting the relation between 
Liouvillian (Perron--Frobenius) and Koopman descriptions of the dynamics.
Exploiting a simple multiplicative property of functions under the action of the Liouvillian and Koopman operators, we explicitly construct a family of functions whose ratios yield the invariants of the underlying dynamics.
Our analysis successfully reproduces the full set of $N-3$ invariants known in Watanabe--Strogatz theory, and offers an alternative spectral perspective.
We demonstrate this approach for a well-studied class of phase models, including the Ermentrout--Kopell, pairwise Kuramoto, and higher-order Kuramoto models. 
\end{abstract}

\maketitle

\section{Introduction}

Self-sustained oscillations are ubiquitous in the real world,
including chemical oscillations~\cite{prigogine1968symmetry,Kuramoto1984chemical,winfree1967phase}, neuronal dynamics~\cite{fitzhugh1955threshold,nagumo1962pulse,ermentrout1986parabolic}, 
circadian rhythms~\cite{Goldbeter1995circadian,kori2017jetlag}, and
Josephson junction arrays~\cite{Watanabe1994constants}.
These oscillations emerge from various elementary processes, but they can often be 
reduced to one-dimensional dynamics by focusing on their oscillation phases~\cite{winfree1967phase,Kuramoto1984chemical,sakaguchi1986soluble,daido1992order,mauroy2013isostables,nakao2016phase,shirasaka2017phase,yawata2024autoencoder,namura2026arbitrary,ozawa2026delay}. 
 
A mathematical model that exclusively focuses on the phase dynamics of an oscillator is called a \textit{phase model}~\cite{winfree1967phase}.
In this study, we focus on a well-studied class of phase models represented in the following form:
\begin{subequations}\label{eq:base_equation}
    \begin{align}
        \frac{d\theta}{dt} &= f(t) + g(t) \cos \theta + h(t) \sin \theta,\label{eq:time_dependent_equation}
\end{align}
where $\theta(t) \in [0, 2\pi]$ is the oscillator phase ($0$ and $2\pi$ are identified) at time $t$ and $f(t)$, $g(t)$, and $h(t)$ are smooth functions of $t$.
This class of systems has been studied extensively since the pioneering work by Watanabe and Strogatz~\cite{Watanabe1993integrability,Watanabe1994constants}.
We are typically interested in a population of globally coupled oscillators, where each oscillator obeys Eq.~\eqref{eq:time_dependent_equation} and interacts with other oscillators through the functions $f$, $g$, and $h$.
The equation for the phase $\theta_j \in [0, 2\pi]$ of the $j$th oscillator ($j=1, \ldots, N$), where $N$ is the total number of oscillators, is then explicitly given by
\begin{align}
        \frac{d\theta_j}{dt} &= f(\bm{\theta}_N, t) + g(\bm{\theta}_N, t) \cos \theta_j + h(\bm{\theta}_N, t) \sin \theta_j,
   \label{eq:state_dependent_equation}
    \end{align}
\end{subequations}
where $\bm{\theta}_N = (\theta_1, \theta_2, \ldots, \theta_N)$. 
Well-known examples of globally coupled oscillators described by Eq.~\eqref{eq:state_dependent_equation} include the Winfree and Kuramoto models~\cite{winfree1967phase,Kuramoto1984chemical}, 
which have been extensively studied to investigate collective synchronization phenomena in coupled-oscillator systems.

For $N$ identical oscillators, 
both Eq.~\eqref{eq:time_dependent_equation} for 
independent oscillators and Eq.~\eqref{eq:state_dependent_equation} for 
coupled oscillators
admit at least $N-3$ invariants, i.e., constants of motion 
or conserved quantities, as shown by Watanabe and Strogatz in~\cite{Watanabe1993integrability}. 
These invariants play an essential role in reducing the dimensionality and in establishing the integrability of the system described by Eq.~\eqref{eq:time_dependent_equation} or~\eqref{eq:state_dependent_equation}.
Using these invariants, the $N$-dimensional dynamics of Eq.~\eqref{eq:time_dependent_equation} or~\eqref{eq:state_dependent_equation} can be
reduced to a system with only three degrees of freedom via the Watanabe--Strogatz transformation~\cite{Watanabe1993integrability},
which was later shown to be a type of M\"{o}bius transformation~\cite{Marvel2009identical}.
This Watanabe--Strogatz framework for dimensionality reduction has since been extended to phase-oscillator systems consisting of subpopulations~\cite{Pikovsky2008partially,hong2011conformists}, under periodic forcing~\cite{atsumi2012persistent}, high-dimensional Kuramoto models~\cite{lohe2018vector}, kinetic vector models~\cite{park2022kinetic}, and higher-order coupled models~\cite{jain2025higher}.
In addition, these invariants are also analyzed from the perspective of the Riccati equation~\cite{reid1972riccati} associated with Eq.~\eqref{eq:base_equation}~\cite{Marvel2009identical,cestnik2024integrability,pazo2025spiking}. 

In this study, we explore the invariants of the above class of coupled-oscillator systems by using the eigenfunctions of the Koopman and Perron--Frobenius operators, which are adjoint to each other.
The Koopman operator~\cite{Koopman1931hamiltonian,Koopman1932spectra,lasota2013chaos,mezic2005spectral,budivsic2012applied,mezic2013analysis,mauroy2020koopman,brunton2022modern} is a linear operator that governs the time evolution of observables for dynamical systems, while the Perron--Frobenius operator~\cite{mezic2005spectral,Gaspard_1998,lasota2013chaos,brunton2022modern} is a linear operator describing the evolution of the state density of dynamical systems.
This framework enables the study of nonlinear systems through linear techniques, such as spectral analysis. For example, the dynamical properties of a system, including its invariants, can be characterized by the eigenfunctions of the Koopman operator~\cite{budivsic2012applied,mauroy2013isostables,shirasaka2017phase,taga2021koopman,parker2023koopman,taga2024dynamic}.

Koopman operator theory has recently been applied to analyze coupled phase oscillators
in several studies
~\cite{susuki2016power,hu2020synchronization,wang2021probing,mihara2022basin,thibeault2025kuramoto}. 
In this study, we investigate the invariants of the systems given by Eq.~\eqref{eq:time_dependent_equation} or~\eqref{eq:state_dependent_equation}
by exploiting a simple relation that holds between the Perron--Frobenius and Koopman eigenfunctions.
We present a new path to deriving Koopman eigenfunctions using Perron--Frobenius eigenfunctions, which provides 
an alternative viewpoint on the invariants of globally coupled phase oscillators.

\section{A general relation between Perron--Frobenius and Koopman eigenfunctions}
\subsection{Preliminaries}

For a continuous dynamical system of the form 
\begin{align}
    \label{eq:ode} 
    \frac{d\bm{x}}{dt} = \bm{F}(\bm{x}, t),
\end{align}
where $\bm{x}(t) \in {\mathbb R}^N$ is the state of the system at time $t$ and ${\bm F}(\bm{x}, t) \in {\mathbb R}^N$ is the vector field representing the dynamics, the continuous evolution of an observable $u(\bm{x}) : {\mathbb R}^N \to {\mathbb C}$ 
is given by
\begin{align}
    \label{eq:Koopman_operator}
   \frac{d}{dt}u (\bm{x}) =  (\mathcal{K} u) (\bm{x}) = \bm{F}(\bm{x}, t) \cdot \nabla u(\bm{x}),
\end{align}
where $\mathcal{K} = \bm{F}(\bm{x}, t) \cdot \nabla$ is the infinitesimal generator of the Koopman operator (Koopman generator)~\cite{brunton2022modern} and $\nabla = \partial / \partial \bm{x}$ represents the gradient operator.
The infinitesimal generator $\mathcal{P}$ of the Perron--Frobenius (PF) operator for continuous-time systems is the adjoint of $\mathcal{K}$,
\begin{align}
    \label{eq:PF_operator}
    \mathcal{P}v(\bm{x}) = -\nabla \cdot \left( \bm{F}(\bm{x}, t) v(\bm{x}) \right) 
\end{align}
which is also known as the Liouville operator~\cite{Gaspard_1998,kato2021asymptotic, brunton2022modern}.
Here, the adjoint is defined with respect to the $L^2$ inner product of two complex functions $u(\bm{x})$ and $v(\bm{x})$ on the state space (e.g., on the $N$-torus when ${\bm x}$ represents the phase variables),
\begin{align}
    \langle u,\mathcal{K}v\rangle = \langle \mathcal{P}u,v\rangle,\quad \langle u, v\rangle = \int u(\bm{x}) \overline{v(\bm{x})}\ d\bm{x},
\end{align}
where the overline indicates complex conjugate.
While $\mathcal{K}$ describes the time evolution of an observable function, 
$\mathcal{P}$ has the physical meaning that it describes the time evolution of 
a probability density function $\rho({\bm x}, t)$ of the system state that is transported by the flow generated by Eq.~\eqref{eq:ode},
which obeys the Liouville equation,
\begin{align}
 \label{eq:continuity_eq}
 \frac{\partial}{\partial t}\rho({\bm x}, t) = \mathcal{P}\rho({\bm x}, t).
 \end{align}
We note that the Koopman and PF operators 
can generally be defined in the 
discrete-time setting, but we consider only their continuous-time generators in this study.

In what follows, we refer to the eigenfunction $u_\lambda(\bm{x})$ and eigenvalue $\lambda \in \mathbb{C}$ of $\mathcal{K}$ satisfying $\mathcal{K} u_\lambda(\bm{x}) = \lambda u_\lambda(\bm{x})$ as the \textit{Koopman eigenfunction} and \textit{eigenvalue}, and the eigenfunction $v_\mu(\bm{x})$ and eigenvalue $\mu \in \mathbb{C}$ of $\mathcal{P}$ satisfying $\mathcal{P} v_\mu(\bm{x}) = \mu v_\mu(\bm{x})$ as the \textit{Perron--Frobenius (PF) eigenfunction} and \textit{eigenvalue}. 
Note that the eigenvalues of $\mathcal{K}$ and $\mathcal{P}$ are related through the adjointness~\cite{brunton2022modern}.
In particular, from Eq.~\eqref{eq:Koopman_operator}, a Koopman eigenfunction with eigenvalue $\lambda = 0$ of $\mathcal{K}$ is an invariant of Eq.~\eqref{eq:ode}.

We begin our analysis of the relationship between the Koopman and PF eigenfunctions for general dynamical systems given in the form of Eq.~\eqref{eq:ode} with the following lemma.
Note that $\mathcal{K}$ and $\mathcal{P}$ are generally $t$-dependent.
\begin{lemma}
    \label{lem:1}
    Let $u_1(\bm{x})$ and 
    $u_2(\bm{x})$ be scalar-valued functions, where 
    $u_2({\bm x}) / u_1({\bm x})$ is finite, and define two functions
    \begin{align}
    \Lambda_1(\bm{x}, t) := \frac{\mathcal{P}u_1(\bm{x})}{u_1(\bm{x})}, \quad \Lambda_2(\bm{x},t) := \frac{\mathcal{P}u_2(\bm{x})}{u_2(\bm{x})}.
    \end{align}
    Then the following identity holds: 
    \begin{align}
        \mathcal{K} \left( \frac{u_2(\bm{x})}{u_1(\bm{x})} \right) = \left\{ \Lambda_1(\bm{x},t) - \Lambda_2(\bm{x},t) \right\} \frac{u_2(\bm{x})}{u_1(\bm{x})}.
    \end{align}
\end{lemma}
\begin{proof} 
   \begin{align}
    \mathcal{K} \left( \frac{u_2}{u_1} \right)
    &= \bm{F}
    \cdot \nabla \left( \frac{u_2}{u_1} \right)  
    = \frac{u_1 \left( \bm{F} \cdot \nabla u_2 \right) - u_2 \left( \bm{F} \cdot \nabla u_1 \right)}{u_1^2} \cr
    &= \frac{1}{ u_1^2} \big[ u_1 \left\{ -\mathcal{P} u_2 - (\nabla\cdot\bm{F})u_2  \right\}  - u_2 \left\{ -\mathcal{P} u_1 - (\nabla\cdot\bm{F})u_1 \right\} \big] \cr
   &= \frac{1}{u_1^2} \big\{ u_1 (-\mathcal{P} u_2) - u_2 (-\mathcal{P} u_1) \big\}      = (\Lambda_1  - \Lambda_2 )\, \frac{u_2}{u_1}.
\end{align} 
\end{proof}
Note that $\Lambda_1$ and $\Lambda_2$ are generally not eigenvalues of $\cal{P}$ as they are point-wise functions of ${\bm x}$.
From Lemma~\ref{lem:1}, we immediately obtain the following theorem. 
\begin{theorem}
    \label{thm:1}
    Let $u_1(\bm{x})$ and $u_2(\bm{x})$ be general scalar-valued functions satisfying
    \begin{align}
    \frac{\mathcal{P}u_1}{u_1} = \frac{\mathcal{P}u_2}{u_2}.
    \end{align}
    Then their ratio 
   ${u_2(\bm{x})} / {u_1(\bm{x})}$ is a Koopman eigenfunction of $\mathcal{K}$ corresponding to the eigenvalue $0$, and thus represents an invariant of the system.
\end{theorem}
\begin{proof} 
Since 
${\mathcal P}u_1 / u_1 = \Lambda_1({\bm x}, t) = {\mathcal P}u_2 / u_2 = \Lambda_2({\bm x}, t)$,
\begin{align}
    \mathcal{K} \left( \frac{u_2}{u_1} \right) = \left\{ \Lambda_1({\bm x}, t) - \Lambda_2({\bm x}, t)\right\} \frac{u_2}{u_1} = 0.
\end{align}
\end{proof}
\begin{theorem}
    \label{thm:2}
    Let $u_1(\bm{x})$ and $u_2(\bm{x})$
      be two PF eigenfunctions of $\mathcal{P}$ 
      and $\lambda_1$ and $\lambda_2$ be the corresponding eigenvalues.
    Then their ratio 
    ${u_2(\bm{x})} / {u_1(\bm{x})}$
     is a Koopman eigenfunction of $\mathcal{K}$ with eigenvalue $\lambda_1 - \lambda_2$.
\end{theorem}
\begin{proof} 
	Since $\mathcal{P} u_1({\bm x}) = \lambda_1 u_1({\bm x})$ and $\mathcal{P} u_2({\bm x}) = \lambda_2 u_2({\bm x})$, $\Lambda_1({\bm x}, t) = \lambda_1$ and $\Lambda_2({\bm x}, t) = \lambda_2$, hence
	\begin{align}
    \mathcal{K} \left( \frac{u_2}{u_1} \right) = ( \lambda_1 - \lambda_2 ) \frac{u_2}{u_1}.
	\end{align} 
\end{proof}

Also, we obtain the following corollary.
\begin{corollary}
    Any ratio of two PF eigenfunctions associated with the same eigenvalue is an invariant of the system.
\end{corollary} 
\begin{proof} 
If we assume $\lambda_1 = \lambda_2$, $\mathcal{K}( u_2 / u_1 ) = 0$ and hence $u_2 / u_1$ is invariant.
\end{proof}

\section{Coupled phase oscillators}

We now focus on a system of identical phase oscillators described by 
Eq.~\eqref{eq:state_dependent_equation}. The PF generator  $\mathcal{P}$ of Eq.~\eqref{eq:state_dependent_equation} acts on a scalar function $u(\bm{\theta}_N, t)$ as
\begin{align}
\mathcal{P}u:= - \sum_{j=1}^{N} \frac{\partial}{\partial \theta_j} \left[ ( f + g \cos \theta_j + h \sin \theta_j ) u \right].
\end{align}
We assume $N \geq 3$ and introduce the following function:
\footnote{We note that $\psi^N_{\bm{q}}\notin L^2([0,2\pi]^N)$ because it diverges when
$\theta_{q^{(j)}}=\theta_{q^{(j+1)}}$,
which can occur when $\theta_i = \theta_j$ for some $(i, j)$. 
In practice, we therefore consider a domain that excludes neighborhoods of these collision sets, for instance $X_\varepsilon
= \left\{(\theta_1,\ldots,\theta_N)\ \middle|\ 
|\theta_i-\theta_j|>\varepsilon \ \text{for all } i\neq j
\right\}$ with some $\varepsilon>0$. For identical oscillators evolving under a smooth vector field, the induced flow is unique and invertible for any finite time; hence phases cannot coincide in finite time unless they coincide initially.
We use $X_\varepsilon$ whenever needed to keep $\psi_{\bm{q}}^N$ well behaved. }
\begin{align}
    \label{eq:eigenfunction}
	\psi_{\bm{q}}^N(\bm{\theta}_N) := \frac{1}{\prod_{j=1}^N \sin\left( \frac{\theta_{q^{(j)}} - \theta_{q^{(j+1)}}}{2} \right)}.
\end{align}
Here, the vector index $\bm{q} = \left(q^{(1)}, q^{(2)}, \ldots, q^{(N)}\right)$ represents a permutation of the oscillator indices $\{1, 2, \ldots, N\}$, where $q^{(N+1)} := q^{(1)}$.
\par

\begin{theorem}
\label{thm:PF_on_psi}
By operating $\mathcal{P}$ on $\psi^N_{\bm{q}}$, the following equation holds:
\begin{align}
	\label{eq:fp11} 
    \mathcal{P} \psi_{\bm{q}}^N(\bm{\theta}_N) = \Lambda(\bm{\theta}_N, t) \psi_{\bm{q}}^N(\bm{\theta}_N),
\end{align}
where
\begin{align}
\label{eq:PF_eigenvalue_equation}
    &\Lambda(\bm{\theta}_N, t)
    := -\sum_{j=1}^N \left[
    \frac{\partial f(\bm{\theta}_N, t)}{\partial \theta_j}
    + \cos \theta_j \frac{\partial g(\bm{\theta}_N, t)}{\partial \theta_j}
    + \sin \theta_j \frac{\partial h(\bm{\theta}_N, t)}{\partial \theta_j}
    \right].
\end{align}
\end{theorem}
\begin{proof}
    See Appendix~A.
\end{proof}

\begin{theorem}
The ratio 
\begin{align}
    \Psi^N_{\bm{q}_1, \bm{q}_2}(\bm{\theta}_N) := \frac{\psi^N_{\bm{q}_2}(\bm{\theta}_N)}{\psi^N_{\bm{q}_1}(\bm{\theta}_N)}
\end{align}
is an eigenfunction of the associated Koopman generator $\mathcal{K}$ corresponding to eigenvalue $0$, 
and thus it serves as an invariant for Eq.~\eqref{eq:base_equation}.
\end{theorem}
\begin{proof}
$\mathcal{P} \psi_{\bm{q}_1}^N / \psi_{\bm{q}_1}^N = \mathcal{P} \psi_{\bm{q}_2}^N / \psi_{\bm{q}_2}^N$ holds for any pair of permutations $\bm{q}_1$ and $\bm{q}_2$ because $\Lambda(\bm{\theta}_N, t)$ in Eq.~\eqref{eq:PF_eigenvalue_equation} does not depend on the permutation 
$\bm{q}$ of the indices $\{1, 2, \ldots, N\}$. From Theorem~\ref{thm:1}, the ratio $\psi^N_{\bm{q}_2} / \psi^N_{\bm{q}_1}$ is an eigenfunction of $\mathcal{K}$ with eigenvalue $0$.
\end{proof}

\par

We have introduced $\psi^N_{\bm{q}}$ for $N\geq 3$, but nontrivial invariants of the form $\Psi^N_{\bm{q}_1,\bm{q}_2}$ arise only for $N\geq 4$. When $N=3$, $\psi^3_{\bm{q}}$ is the same for any permutation $\bm{q}$, thus we have only a single trivial invariant  $\Psi^{3}_{\bm{q}_1,\bm{q}_2} = \pm1$ (the sign depends on whether the permutations $\bm{q}_1$ and $\bm{q}_2$ are even or odd).
For general $N$, the invariant $\Psi^N_{\bm{q}_1, \bm{q}_2}$ can be related to the known invariants, called cross ratios~\cite{Marvel2009identical}, and we can show that the number of independent invariants of the form\footnote{We note that other invariants that do not take the form of the cross ratio can also exist; see the examples of the Theta model in Appendix~C.} $\Psi^N_{\bm{q}_1, \bm{q}_2}$ is $N-3$ except for a constant function, as shown in Appendix~B.

Let us remark on the physical meaning of $\psi^N_{\bm{q}}(\bm{\theta}_N)$ here.
Suppose that $\Lambda(\bm{\theta}_N, t) = \lambda = \text{const.}$. Then, $\psi^N_{\bm{q}}$ itself is a PF eigenfunction of $\mathcal{P}$ with eigenvalue $\lambda$.
Since $\mathcal{P}$ describes the evolution of the Liouville equation~\eqref{eq:continuity_eq},
\begin{align}
    \label{eq:particular solution}
    \rho(\bm{\theta}_N, t) = A e^{\lambda t} \psi^N_{\bm{q}}(\bm{\theta}_N)
\end{align}
is a particular solution of Eq.~\eqref{eq:continuity_eq}, where $A$ is an arbitrary constant.
In particular, for the case of independent oscillators where each oscillator obeys Eq.~\eqref{eq:time_dependent_equation} with $f$, $g$, and $h$ depending only on time $t$, we obtain $\Lambda(\bm{\theta}_N, t) = 0$ in Eq.~\eqref{eq:PF_eigenvalue_equation}. 
Therefore, $\psi^N_{\bm{q}}$ is the PF eigenfunction with eigenvalue $\lambda=0$ and thus it is a stationary solution of the Liouville equation~\eqref{eq:continuity_eq}, which can be interpreted as a stationary (unnormalized) density.

Let us also remark that the modulus $\left|\psi_{\bm{q}}^N\right|$ of $\psi^N_{\bm{q}}$ has the same property as Eq.~\eqref{eq:fp11}, i.e.,
\begin{align}
	\label{eq:fp_mod} 
    \mathcal{P} \left|\psi_{\bm{q}}^N (\bm{\theta}_N) \right| = \Lambda(\bm{\theta}_N, t) \left|\psi_{\bm{q}}^N (\bm{\theta}_N) \right|,
\end{align}
where $\Lambda$ is given in Eq.~\eqref{eq:PF_eigenvalue_equation}. 
Therefore, we can repeat the same argument as above also for $\left|\psi_{\bm{q}}^N\right|$. 
In particular, if $\left|\psi_{\bm{q}}^N\right|$ is a PF eigenfunction of $\mathcal{P}$,
we obtain a particular solution of Eq.~\eqref{eq:continuity_eq} as
\begin{align}
        \rho(\bm{\theta}_N, t) = A e^{\lambda t} \left|\psi^N_{\bm{q}} (\bm{\theta}_N) \right|.
\end{align}
In the numerical simulations presented below, because $\left|\psi^N_{\bm{q}}\right|$ is singular on $[0,2\pi]^N$, we clip it and use the resulting function as a weighting function to sample the initial conditions. We then simulate the evolution of the distribution of system states and compare it with the Liouvillian dynamics of the state density.

In what follows, we present a few explicit examples that have been well studied in the literature. 
We consider the Ermentrout--Kopell Theta model ($N$ oscillators without interaction)~\cite{ermentrout1986parabolic,ermentrout2008ekmodel}, Kuramoto--Sakaguchi model ($N$ oscillators with pairwise interaction)~\cite{Kuramoto1984chemical,sakaguchi1986soluble}, and higher-order (HO) Kuramoto model ($N$ oscillators with 3-body interaction)~\cite{tanaka2011multistable,skardal2020higher,leon2024anomalous,leon2025theory,fujii2025emergence}. In each case, we focus on $\mathcal{P}\psi^N_{\bm{q}}$ to see the difference between the models.
As discussed above, $\Psi^N_{\bm{q}_1, \bm{q}_2}$ is an invariant for each system. 
For each system, we show the evolution of the system states for $N=3$, where initial conditions are sampled according to the weight given by $\left|\psi^3_{\bm{q}}\right|$.
We also show numerical examples of the time evolution of $\bm{\theta}_N$, $\psi^N_{\bm{q}}$, and $\Psi^N_{\bm{q}_1, \bm{q}_2}$ with $N=10$.

\subsection{Example 1. Ermentrout--Kopell Theta model}

The Ermentrout--Kopell Theta model~\cite{ermentrout1986parabolic,ermentrout2008ekmodel} is a phase model of neuronal dynamics and also serves as the normal form of a saddle-node bifurcation on an invariant circle. %
We consider $N$ uncoupled oscillators driven by a common input.
The phase of each oscillator obeys
\begin{align}
    \label{eq:theta_model}
    \frac{d\theta_j}{dt} = 1 - \cos \theta_j + (1 + \cos \theta_j)\, I(t),
\end{align}
for $j=1, ..., N$, where $I(t)$ is the common external input to the oscillator. 
Thus, the Theta model corresponds to the case with $f(t) = 1 + I(t)$, $g(t) = -1 + I(t)$, and $h(t) = 0$ in Eq.~\eqref{eq:base_equation}. 
Since $f$, $g$, and $h$ are independent of the phases ${\bm \theta}_N = (\theta_1, \ldots, \theta_N)$, all derivatives in Eq.~\eqref{eq:PF_eigenvalue_equation} vanish and hence $\Lambda = 0$ holds.
Therefore, $\psi^N_{\bm{q}}({\bm \theta}_N)$ is a PF eigenfunction of this system with eigenvalue $0$. 
By the method explained above, we can introduce $N-3$ invariants $\Psi^N_{\bm{q}_1,\bm{q}_2}({\bm \theta}_N)$ from $\psi^N_{\bm{q}_1}({\bm \theta}_N)$ and $\psi^N_{\bm{q}_2}({\bm \theta}_N)$, where $\bm{q}_1$ and $\bm{q}_2$ are permutations of the oscillator indices $\{ 1, 2, ..., N\}$.

First, we consider $N=3$ oscillators with $I(t) = \sin t$ as an example. Only one independent PF eigenfunction of the form $\psi^3_{\bm q}$ with ${\bm q} = (1, 2, 3)$ exists in this case.
Figure~\ref{fig:theta_model}(a) shows the evolution of the distribution of the phases ${\bm \theta}_3 = (\theta_1, \theta_2, \theta_3)$ on the $(\theta_1-\theta_2) - (\theta_2-\theta_3)$ plane obtained by simulating the Theta model, Eq.~\eqref{eq:theta_model}, from $10^5$ initial conditions. 
Here, the initial conditions were randomly sampled with weights proportional to the PF eigenfunction $\left|\psi^3_{\bm{q}}({\bm \theta}_3)\right|$, which was clipped at $10^2$ because $\left|\psi^3_{\bm{q}}({\bm \theta}_3)\right|$ diverges when $\theta_1 = \theta_2$, $\theta_2 = \theta_3$, or $\theta_3 = \theta_1$.
The stationary solution $\left|\psi^3_{\bm{q}}({\bm \theta}_3)\right|$ of the Liouville equation~\eqref{eq:continuity_eq} given by Eq.~\eqref{eq:eigenfunction} is shown on the same plane (also clipped at $10^2$ and divergent points are also plotted in black). 
In this example, the phase distribution on the $(\theta_1-\theta_2)$--$(\theta_2-\theta_3)$ plane remains approximately stationary and is well captured by the stationary solution of the Liouville equation.

Next, we consider the case with $N=10$ and $I(t) = \sin t$. Although the oscillators are uncoupled, the common time-dependent input entrains them, leading to phase locking with the input.
In this case, we can identify eight independent functions of the form
$\psi^{10}_{\bm q}(\bm{\theta}_{10}(t))$, from which $N-3=7$ independent
invariants of the form $\Psi^{10}_{\bm q_1,\bm q_2}(\bm{\theta}_{10}(t))$
can be constructed.

Figure~\ref{fig:theta_model}(b) shows the time evolution of (b-1) ${\bm \theta}_{10}(t) = (\theta_1(t), \theta_2(t), \ldots, \theta_{10}(t))$,
(b-2) $\left|\psi^{10}_{\bm{q}}(\bm{\theta}_{10}(t))\right|$
, and (b-3) $\left|\Psi^{10}_{\bm{q}_1,\bm{q}_2}(\bm{\theta}_{10}(t))\right|$.
We can confirm that the ratio $\left|\Psi^{10}_{\bm{q}_1,\bm{q}_2}({\bm \theta}_{10}(t))\right|$ remains invariant under the non-stationary evolution of the phases ${\bm \theta}_{10}(t) = (\theta_1(t), \theta_2(t), \ldots, \theta_{10}(t))$.
\begin{figure}[h]
\includegraphics[width=\linewidth]{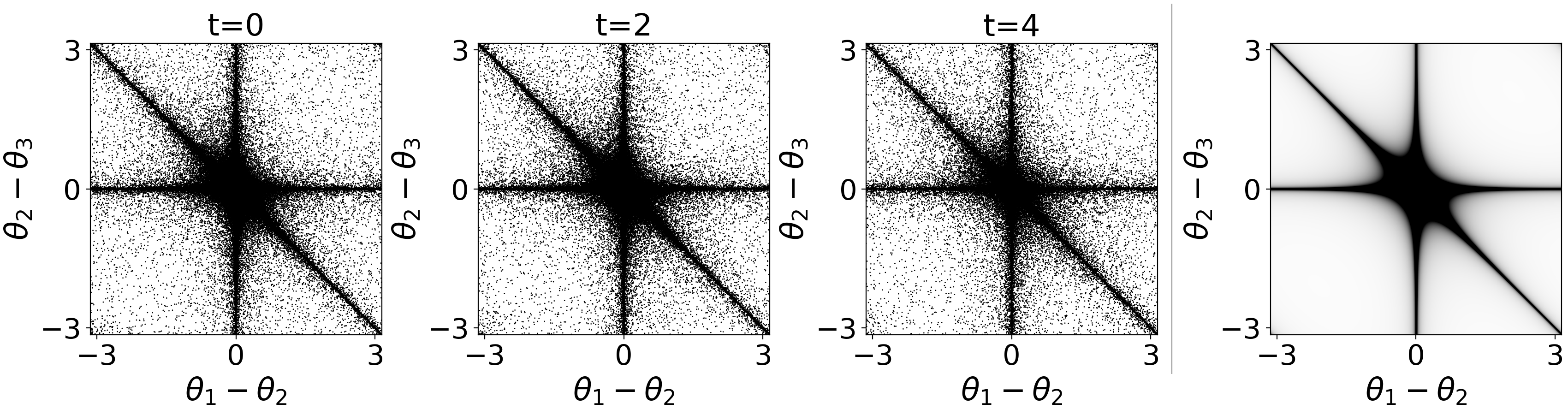}\\
    \hspace{50mm}(a-1)\hspace{62mm}(a-2)\\
\includegraphics[width=\linewidth]{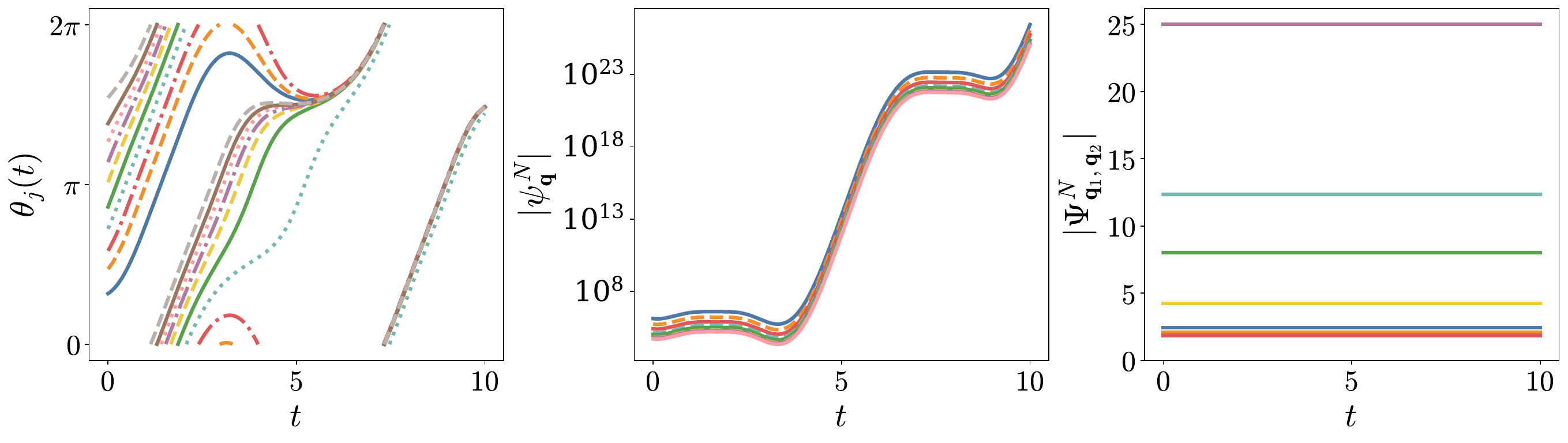}\\
   \hspace{23mm}(b-1)\hspace{40mm}(b-2)\hspace{38mm}(b-3)
    \caption{Theta model with $I(t) = \sin t$. 
    (a) 
    Results for $N=3$ oscillators.
    (a-1) 
    Evolution of the distribution of the phases ${\bm \theta}_3 = (\theta_1, \theta_2, \theta_3)$ on the $(\theta_1-\theta_2) - (\theta_2-\theta_3)$ plane, obtained by simulating the Theta model from $10^5$ initial conditions sampled randomly with the weights proportional to $\left|\psi^3_{\bm{q}}({\bm \theta}_3)\right|$ (clipped at $10^2$). Each black point corresponds to a point ${\bm \theta}_3$ in the state space.
    (a-2) 
	Density plot of $\left|\psi^3_{\bm{q}}({\bm \theta}_3)\right|$ on the same plane, which is the stationary solution of the Liouville equation, clipped at $10^2$.
    Vertical, horizontal, and diagonal lines correspond to $\theta_1=\theta_2$, $\theta_2=\theta_3$, and $\theta_3=\theta_1$, which are singular points of $\left|\psi^3_{\bm{q}}({\bm \theta}_3)\right|$ and also plotted in black.
    (b)  
	Results for $N=10$ oscillators.
	Typical time evolution of the phases ${\bm \theta}_{10}(t) = (\theta_1(t), \theta_2(t), \ldots, \theta_{10}(t))$ (b-1), $\left|\psi^{10}_{\bm{q}}(\bm{\theta}_{10}(t))\right|$ (b-2), and $\left|\Psi^{10}_{\bm{q}_1,\bm{q}_2}(\bm{\theta}_{10}(t))\right|$ (b-3) with time $t$.
    }
    \label{fig:theta_model}
\end{figure}

\subsection{Example 2. Kuramoto--Sakaguchi Model}

The Kuramoto--Sakaguchi model~\cite{Kuramoto1984chemical,sakaguchi1986soluble} describes a system of globally coupled oscillators. When all oscillators share the same frequency, the model is given by
\begin{align}
    \label{eq:kuramoto_model}
    \frac{d\theta_j}{dt} &= \omega + \frac{K}{N}\sum_{k=1}^{N}\sin(\theta_k - \theta_j + \delta), \quad j = 1, 2, \ldots, N,
\end{align}
where the parameters $\omega$, $K$, and $\delta$ are the natural frequency, coupling intensity, and coupling phase lag.

By transforming the right-hand side of Eq.~\eqref{eq:kuramoto_model}, we obtain
\begin{align}
    &\frac{d\theta_j}{dt}
    = \omega + \frac{K}{N} \sum_{k=1}^N \left[
        \sin(\theta_k + \delta)\cos \theta_j
        - \cos(\theta_k + \delta)\sin \theta_j
    \right].
\end{align}
Thus, it has the form of Eq.~\eqref{eq:state_dependent_equation} with
\begin{align}
    f(\bm{\theta}_N, t) =& \omega, \cr
    g(\bm{\theta}_N, t) =& \frac{K}{N} \sum_{k=1}^N \sin(\theta_k + \delta), \cr
    h(\bm{\theta}_N, t) =& -\frac{K}{N} \sum_{k=1}^N \cos(\theta_k + \delta).
\end{align}
From these relations and Theorem~\ref{thm:PF_on_psi}, we obtain
\begin{align}
    \Lambda = - K \cos \delta,
\end{align}
which is a constant that does not depend on $\bm{\theta}_N$ or $t$. Thus, $\psi_{\bm{q}}^N$ is a PF eigenfunction with eigenvalue $-K \cos \delta$. 
\par

In particular, the case $\delta = \frac{\pi}{2}$ (or equivalently $-\frac{\pi}{2}$) is known to be integrable~\cite{Watanabe1993integrability}.
In this case, the phase coupling between the oscillators is neutral and mutual synchronization does not occur for general $N$.
Correspondingly, the PF eigenvalue vanishes, i.e., $\Lambda = - K \cos \delta = 0$, and $\left|\psi^N_{\bm{q}}({\bm \theta}_N)\right|$ corresponds to a stationary solution of the Liouville equation.
In this case, any constant function $c$ is also a PF eigenfunction with eigenvalue $\Lambda = 0$, because
\begin{align}
    \mathcal{P}c &= \sum_{j=1}^N\frac{\partial}{\partial\theta_j} \left[\omega+\frac{K}{N}\sum_{k=1}^N\cos(\theta_k-\theta_j)\right] c 
    = c \frac{K}{N}\sum_{j=1}^N\sum_{k=1}^N[-\sin(\theta_k-\theta_j)]=0.
\end{align}
The ratio of $\psi^N_{\bm{q}}({\bm \theta}_N)$ to a constant function $c$ is a Koopman eigenfunction with eigenvalue $0$, according to Theorem~\ref{thm:2}.
Thus, by assuming $c=1$, $\psi^N_{\bm{q}}({\bm \theta}_N)$ itself is found to be a Koopman eigenfunction.
This invariant has previously been found heuristically in~\cite{Watanabe1993integrability}. Our framework provides an explanation for why such invariants emerge specifically when $\delta = \pm \frac{\pi}{2}$; the constant function becomes a PF eigenfunction only in this case.

Figure~\ref{fig:Kuramoto--Sakaguchi_model} shows the numerical results for $N=3$ oscillators with $K = 1$, $\omega = 0$, and $\delta = \frac{\pi}{2}$.
In Fig.~\ref{fig:Kuramoto--Sakaguchi_model}(a-1), evolution of the distribution of the phases ${\bm \theta}_3 = (\theta_1, \theta_2, \theta_3)$ is plotted on the $(\theta_1-\theta_2) - (\theta_2-\theta_3)$ plane as before, which was obtained by simulating the Kuramoto--Sakaguchi model from $10^5$ initial conditions.
The initial conditions were randomly sampled with weights proportional to $\left|\psi^3_{\bm{q}}({\bm \theta}_3)\right|$ (clipped at $10^2$).
Since $\Lambda = 0$ in this case, the phase distribution remains approximately stationary, and the oscillators do not tend to synchronize.
In Fig.~\ref{fig:Kuramoto--Sakaguchi_model}(a-2), the PF eigenfunction $|\psi^3_{\bm{q}}({\bm \theta}_3)|$, which is a stationary solution of the Liouville equation, is plotted (also clipped at $10^2$).
The vertical, horizontal, and diagonal lines correspond to $\theta_1 = \theta_2$, $\theta_2 = \theta_3$, and $\theta_3 = \theta_1$, respectively, which are singular points of $\left|\psi^3_{\bm{q}}({\bm \theta}_3)\right|$ and plotted in black.
The phase distributions in (a-1) agree well with the stationary state density in (a-2).

Figure~\ref{fig:Kuramoto--Sakaguchi_model}(b) shows the results for $N=10$ oscillators, where
time evolution of (b-1) the oscillator phases ${\bm \theta}_{10}(t) = (\theta_1(t), ..., \theta_{10}(t))$,
(b-2) $\left|\psi^{10}_{\bm{q}}({\bm \theta}_{10}(t))\right|$, and (b-3) time evolution of the ratio $\left|\Psi^{10}_{\bm{q}_1,\bm{q}_2}({\bm \theta}_{10}(t))\right|$
are plotted with respect to $t$.
We can confirm that the functions $\left|\psi^{10}_{\bm{q}}({\bm \theta}_{10}(t))\right|$,
which are also Koopman eigenfunctions, and $\left|\Psi^{10}_{\bm{q}_1,\bm{q}_2}({\bm \theta}_{10}(t))\right|$ remain constant under the non-stationary evolution of ${\bm \theta}_{10}(t)$.
Next, we consider the case $\delta = 0$, where Eq.~\eqref{eq:kuramoto_model} is the standard Kuramoto model with a homogeneous frequency. 
In this case, all oscillator asymptotically synchronize completely, namely, $\lim_{t \to \infty} ( \theta_i(t) - \theta_j(t) ) = 0$ for all $i$ and $j$ from almost all initial conditions.
We note that unstable phase-locked states also exist, which do not converge to complete synchronization.

First, we consider $N=3$ oscillators with $K = 1$, $\omega = 0$, and $\delta = 0$.
Figure~\ref{fig:Kuramoto--Sakaguchi_model}(c-1) shows
the evolution of the distribution of the phases ${\bm \theta}_3 = (\theta_1, \theta_2, \theta_3)$ on the $(\theta_1-\theta_2) - (\theta_2-\theta_3)$ plane as before, obtained by simulating the Kuramoto--Sakaguchi model, Eq.~\eqref{eq:kuramoto_model}, from $10^5$ initial conditions. Here, the initial conditions were randomly sampled with weights proportional to the PF eigenfunction $\left|\psi^3_{\bm{q}}({\bm \theta}_3)\right|$ (clipped at $10^2$).
Since the oscillators tend to synchronize completely from almost all conditions in this case, the phase distribution tends to concentrate at the origin where $\theta_1 = \theta_2 = \theta_3$ as $t$ increases. 
The vertical, horizontal, and diagonal lines correspond to the states where 
two oscillators are synchronized ($\theta_1 = \theta_2$, $\theta_2 = \theta_3$, or $\theta_3 = \theta_1$). Note that there are also unstable fixed points that do not converge to complete synchronization on these lines.

Figure~\ref{fig:Kuramoto--Sakaguchi_model}(c-2) plots the evolution of the state density by the Liouville equation~\eqref{eq:continuity_eq}, where the initial density is taken as the PF eigenfunction $\left|\psi^3_{\bm{q}}({\bm \theta}_3)\right|$. The solution is therefore given by $Ae^{-Kt}\left|\psi^3_{\bm{q}}({\bm \theta}_3)\right|$, which decays exponentially with time $t$. 
The phase distribution in (c-1) and the state density in (c-2) agree well with each other.
As the oscillators synchronize mutually, the state density decays to zero at every point in the state space, except on the vertical, horizontal, or diagonal line where
$\theta_i = \theta_j$ for any pair of $i, j = 1, 2, 3$.
Correspondingly, the PF eigenfunction $\psi^N_{\bm{q}}$ diverges when $\theta_i = \theta_j$ for any $i, j$.
The state density tends to accumulate on these singular sets, which corresponds to the formation of synchronized clusters (i.e., at least some oscillators sharing the same phase).
Then, the state density eventually accumulates at the origin, approaching the completely synchronized state. Note that there also exist singular points corresponding to the unstable solution that do not converge to full synchrony.
Figure~\ref{fig:Kuramoto--Sakaguchi_model}(d) shows the time evolution of (d-1) ${\bm \theta}_{10}(t) = (\theta_1(t), \theta_2(t), \ldots, \theta_{10}(t))$,
(d-2) $|\psi^{10}_{\bm{q}}(\bm{\theta}_{10}(t))|$
, and (d-3) $|\Psi^{10}_{\bm{q}_1,\bm{q}_2}(\bm{\theta}_{10}(t))|$.
We can confirm that the ratio $|\Psi^{10}_{\bm{q}_1,\bm{q}_2}({\bm \theta}_{10}(t))|$ remains invariant under the non-stationary evolution of the phases ${\bm \theta}_{10}(t) = (\theta_1(t), \theta_2(t), \ldots, \theta_{10}(t))$.

\begin{figure}
\includegraphics[width=\linewidth]{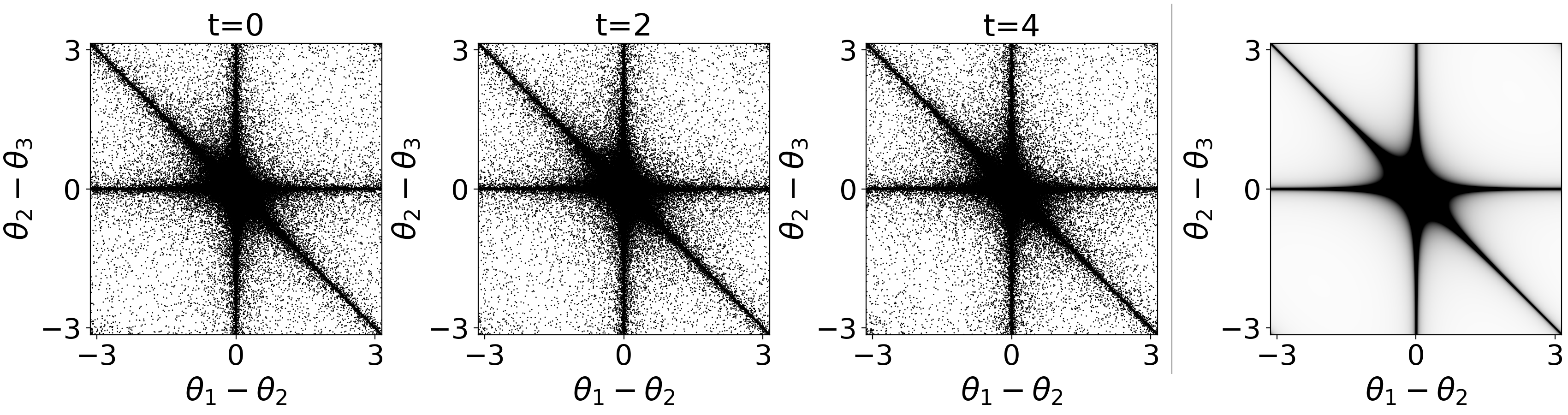}\\
\hspace{54mm}(a-1)\hspace{84mm}(a-2)\\
  \includegraphics[width=\linewidth]{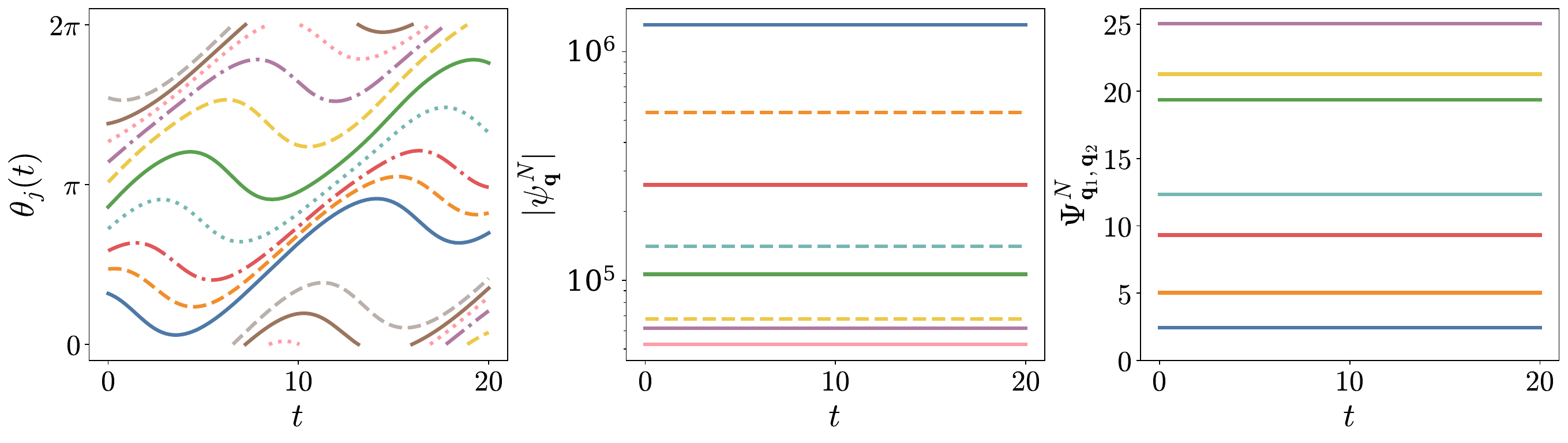}\\
\hspace{12mm}(b-1)\hspace{55mm}(b-2)\hspace{52mm}(b-3)\\
\includegraphics[width=0.49\linewidth]{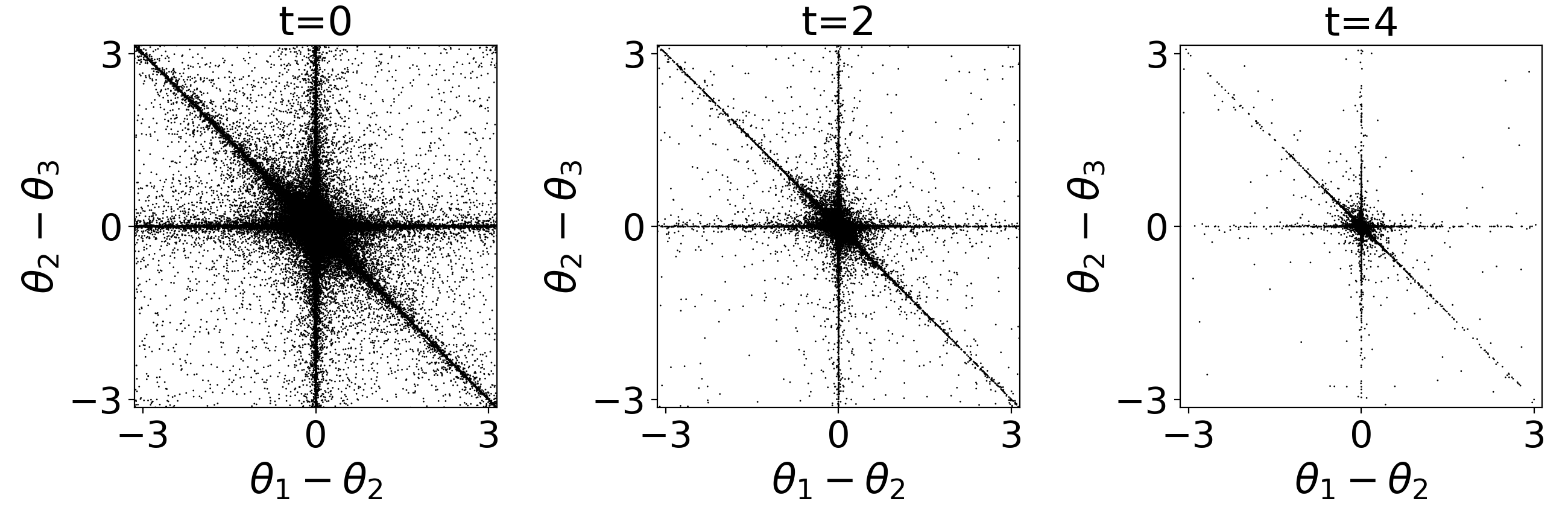}
\vrule width 1pt
\includegraphics[width=0.49\linewidth]{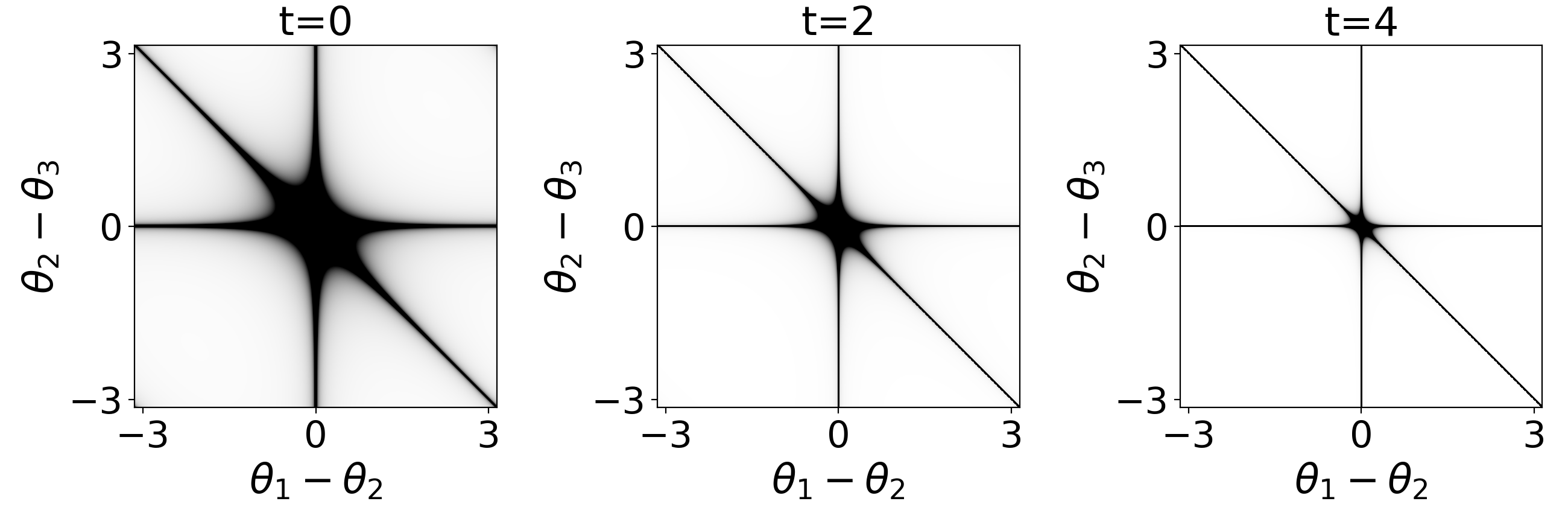}\\
\hspace{6mm}(c-1)\hspace{83mm}(c-2)\\
\includegraphics[width=\linewidth]{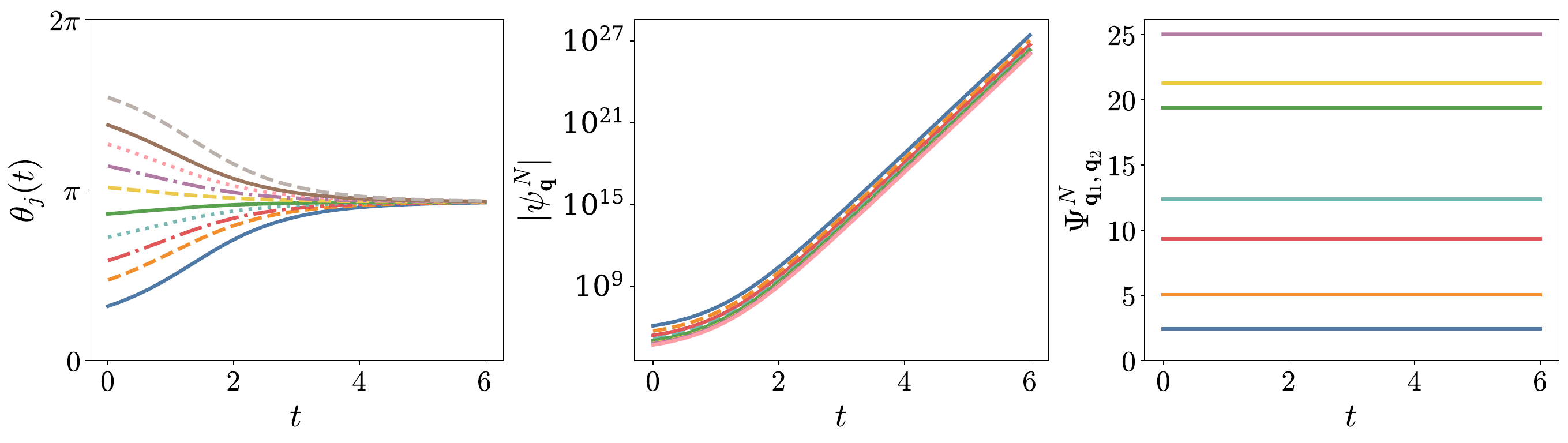}\\
\hspace{10mm}(d-1)\hspace{55mm}(d-2)\hspace{53mm}(d-3)
    \caption{Kuramoto--Sakaguchi model with $K = 1$ and $\omega = 0$. 
    (a) Results for $N=3$ oscillators with $\delta = \pi/2$.
    (a-1) Evolution of the distribution of the phases ${\bm \theta}_3 = (\theta_1, \theta_2, \theta_3)$ on the $(\theta_1-\theta_2) - (\theta_2-\theta_3)$ plane, obtained by simulating the Kuramoto--Sakaguchi model from $10^5$ initial conditions randomly sampled as before.
	(a-2) Stationary state density $\left|\psi^3_{\bm q}({\bm \theta}_3)\right|$ (clipped at $10^2$) of the Liouville equation.
    (b) Results for $N=10$ oscillators with $\delta = \pi/2$.
    Evolution of the phases ${\bm \theta}_{10}(t) = (\theta_1(t), \theta_2(t), \ldots, \theta_{10}(t))$ (b-1),
    $\left|\psi^{10}_{\bm{q}}({\bm \theta}_{10}(t))\right|$  
    (b-2),
    and  $\left|\Psi^{10}_{\bm{q}_1,\bm{q}_2}({\bm \theta}_{10}(t))\right|$
    (b-3) with time $t$.
        (c) Results for $N=3$ oscillators with $\delta = 0$.
    (c-1) 
    Evolution of the distribution of the phases ${\bm \theta}_3 = (\theta_1, \theta_2, \theta_3)$ obtained by simulating the Kuramoto--Sakaguchi model from $10^5$ initial conditions randomly sampled with weights proportional to $\left|\psi^3_{\bm{q}}({\bm \theta}_3)\right|$ (clipped at $10^2$).
    Each black point corresponds to a point ${\bm \theta}_3$ in the state space.
    (c-2) Evolution of the solution $\left|\psi^3_{\bm{q}}({\bm \theta}_3)\right|e^{-Kt}$ (clipped at $10^2$) of the Liouville equation. On the vertical, horizontal, and diagonal lines satisfying $\theta_1 = \theta_2$, $\theta_2 = \theta_3$, and $\theta_3 = \theta_1$, respectively, the solution diverges but is also plotted in black.
    (d) Results for $N=10$ oscillators with $\delta = 0$.
    Evolution of the phases ${\bm \theta}_{10}(t)= (\theta_1(t), \theta_2(t), \ldots, \theta_{10}(t))$ (d-1),
    $\left|\psi^{10}_{\bm{q}}({\bm \theta}_{10}(t))\right|$ for  
    (d-2),
    and $\left|\Psi^{10}_{\bm{q}_1,\bm{q}_2}({\bm \theta}_{10}(t))\right|$
    (d-3) with time $t$.
    }
    \label{fig:Kuramoto--Sakaguchi_model}
\end{figure}

\subsection{Example 3. Higher-order Kuramoto model}

As the last example, we consider a higher-order Kuramoto model.
Recently, systems with higher-order interactions have gained significant attention due to their rich dynamical behaviors beyond pairwise interactions~\cite{tanaka2011multistable,skardal2020higher,leon2024anomalous,leon2025theory,jain2025higher,fujii2025emergence}. 
Among several models with higher-order interactions, we consider the following Kuramoto-type model~\cite{skardal2020higher,leon2024anomalous,leon2025theory}:
\begin{align}
\label{eq:higher_order}
    \frac{d\theta_j}{dt} = \omega + \frac{K}{N^2} \sum_{k,l=1}^{N} \sin(2\theta_k - \theta_l - \theta_j + \delta).
\end{align}
The right-hand side of this equation can be rewritten as
\begin{align}
\frac{d\theta_j}{dt} 
    &= \omega + \frac{K}{N^2} \sum_{k,l=1}^{N} \big[
        \sin(2\theta_k - \theta_l + \delta) \cos \theta_j
        - \cos(2\theta_k - \theta_l + \delta) \sin \theta_j
    \big].
\end{align}
This corresponds to the following case in Eq.~\eqref{eq:state_dependent_equation}:
\begin{align}
    f(\bm{\theta}_N, t) =& \omega, \cr
    g(\bm{\theta}_N, t) =& \frac{K}{N^2} \sum_{k,l=1}^N \sin(2\theta_k - \theta_l + \delta), \cr
    h(\bm{\theta}_N, t) =& -\frac{K}{N^2} \sum_{k,l=1}^N \cos(2\theta_k - \theta_l + \delta).
\end{align}
Using the above relations, we obtain
\begin{align}
    \Lambda(\bm{\theta}_N) = - 2K R_1^2(\bm{\theta}_N) \cos \delta + K R_2^2(\bm{\theta}_N) \cos \delta,
\end{align}
where $R_1(\bm{\theta}_N)$ is the Kuramoto order parameter~\cite{Kuramoto1984chemical} and $R_2(\bm{\theta}_N)$ is the second order parameter~\cite{daido1992order} defined as
\begin{align}
    R_1(\bm{\theta}_N) = \left|\frac{1}{N}\sum_{j=1}^{N}e^{i\theta_j}\right|,\quad R_2(\bm{\theta}_N) = \left|\frac{1}{N}\sum_{j=1}^{N}e^{i2\theta_j}\right|.
\end{align}
These quantities characterize the degree of synchrony and two-cluster formation in the oscillator population, and they are time-dependent functions in general. 
Therefore, for general $\delta$, $\Lambda$ is not constant and $\psi^N_{\bm{q}}$ is not a PF eigenfunction.
However, the ratio $\Psi^N_{\bm{q}_1,\bm{q}_2}$ still forms an invariant by Theorem~\ref{thm:1}.

Only when $\delta = \pm \frac{\pi}{2}$, $\psi^N_{\bm{q}}$ becomes the PF eigenfunction with eigenvalue $0$, namely, it gives a stationary distribution. 
However, in contrast to the Kuramoto--Sakaguchi case, a constant function $c$ is not a PF eigenfunction of this model.
Indeed, introducing
\begin{align}
F_j(\bm{\theta}_N)
:= \omega + \frac{K}{N^2}\sum_{k,l=1}^{N}\sin(2\theta_k-\theta_l-\theta_j+\delta),
\end{align}
we obtain 
\begin{align}
\mathcal{P}c
= -\sum_{j=1}^{N}\frac{\partial}{\partial\theta_j}\bigl(F_j(\bm{\theta}_N)c\bigr)
= -c\sum_{j=1}^{N}\frac{\partial F_j(\bm{\theta}_N)}{\partial\theta_j}.
\end{align}
Thus, it follows that $\mathcal{P}c = \Lambda_c(\bm{\theta}_N)c$ with
\begin{align}
\Lambda_c(\bm{\theta}_N)
= \frac{K}{N^2}\sum_{j,k,l=1}^{N}\cos(2\theta_k-\theta_l-\theta_j+\delta),
\end{align}
which generally depends on $\bm{\theta}_N$. Hence $c$ is not a PF eigenfunction.
Figure~\ref{fig:higher-order_model}(a) and (b) show the results for $\delta = \frac{\pi}{2}$.
In Fig.~\ref{fig:higher-order_model}(a-1), evolution of the distribution of the phases is plotted for $N=3$, which is obtained by simulating the higher-order Kuramoto model from $10^5$ initial conditions randomly sampled with weights proportional to the PF eigenfunction $\left|\psi^3_{\bm{q}}({\bm \theta}_3)\right|$ (clipped at $10^2$). Since $\Lambda = 0$ in this case, the distribution remains approximately stationary.
\par

Figure~\ref{fig:higher-order_model}(a-2) shows the stationary solution of the Liouville equation for $N=3$,
$\psi^3_{\bm{q}}({\bm \theta}_3)$.
This function is divergent on the horizontal, vertical, and diagonal lines, but is plotted in black as before.
The phase distribution obtained by simulating the higher-order Kuramoto model is in good agreement with the stationary solution $\left|\psi^3_{\bm {q}}(\bm{\theta}_{3})\right|$ of the corresponding Liouville equation.
Figure~\ref{fig:higher-order_model}(b) shows the results for $N=10$, where time evolution of 
phases $\bm{\theta}_{10} = (\theta_1, \theta_2, \ldots, \theta_{10})$ (b-1),
two independent PF eigenfunctions $\left|\psi^{10}_{\bm{q}}(\bm{\theta}_{10}(t))\right|$ 
(b-2),
and the ratio $\left|\Psi^{10}_{\bm{q}_1,\bm{q}_2}(\bm{\theta}_{10}(t))\right|$ (b-3).
The ratio $\Psi^{10}_{\bm{q}_1,\bm{q}_2}(\bm{\theta}_{10}(t))$ remains invariant under the non-stationary evolution of the oscillator phases.
We also show the results for $\delta = 0$, where the oscillators are asymptotically synchronized in Fig.~\ref{fig:higher-order_model}(c). In this case, we see that $\left|\psi^{10}_{\bm{q}}(\bm{\theta}_{10}(t))\right|$ is not constant but the ratio $\Psi^{10}_{\bm{q}_1,\bm{q}_2}(\bm{\theta}_{10}(t))$ remains invariant.

\begin{figure}[h]
\includegraphics[width=\linewidth]{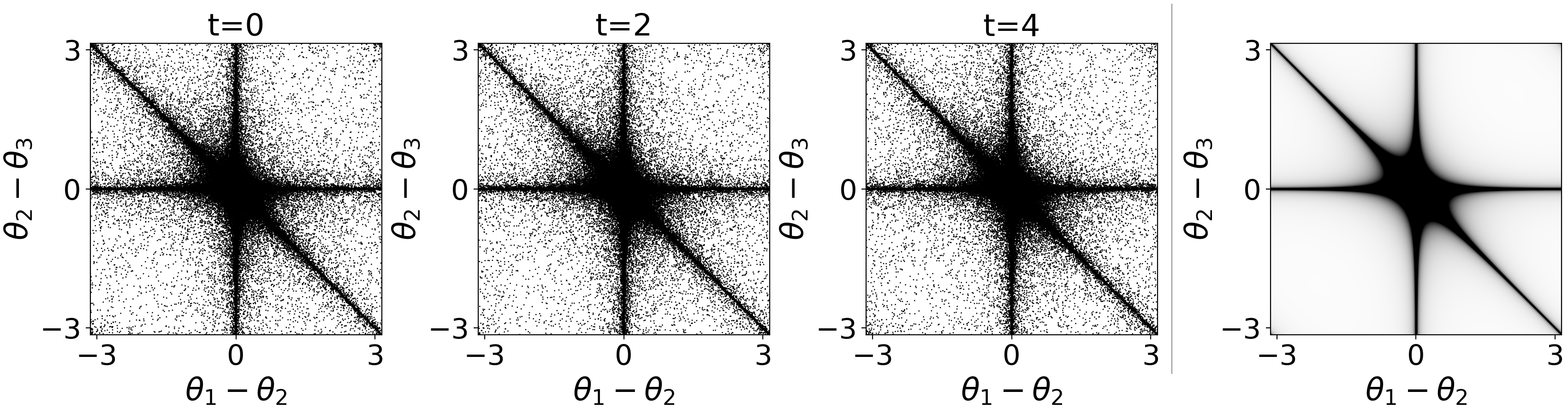}\\
  \hspace{55mm}(a-1)\hspace{85mm}(a-2)\\
  \includegraphics[width=\linewidth]{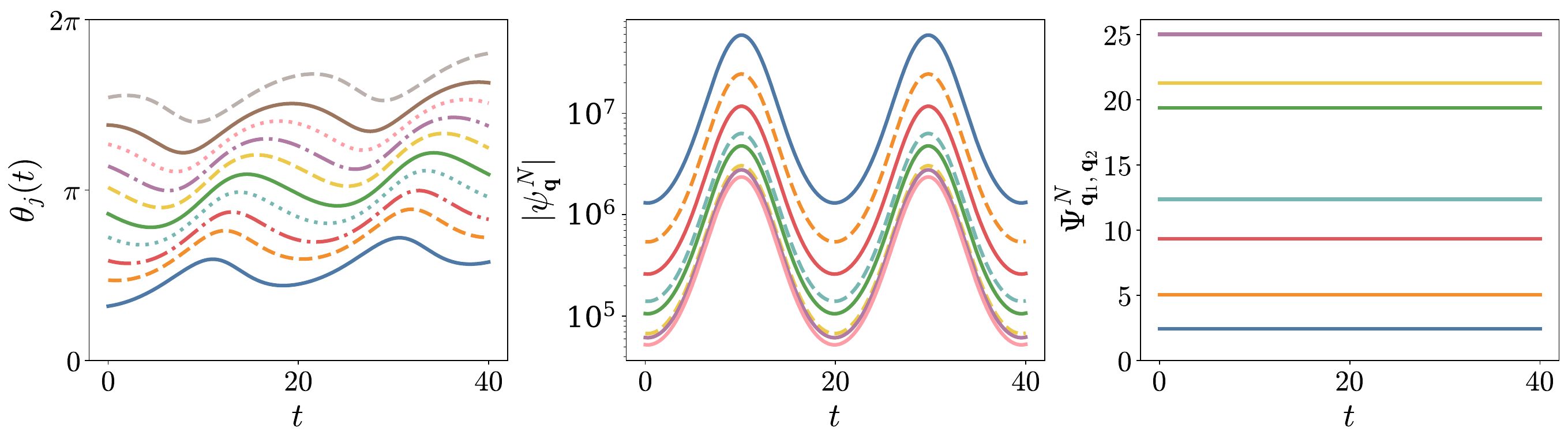}\\
\hspace{10mm}(b-1)\hspace{55mm}(b-2)\hspace{52mm}(b-3)\\
  \includegraphics[width=\linewidth]{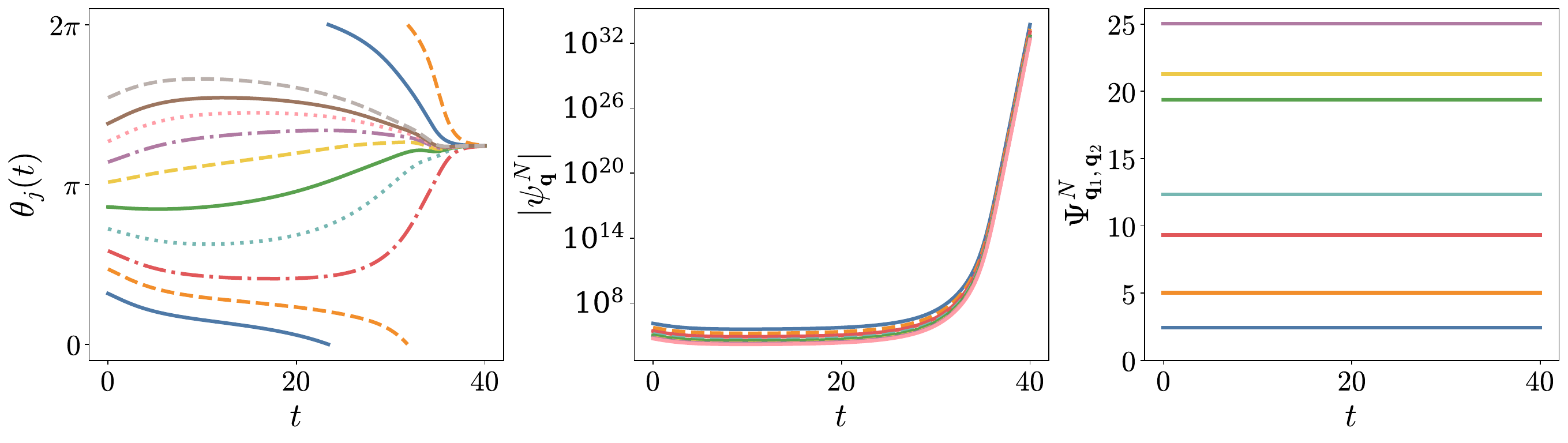}\\
\hspace{10mm}(c-1)\hspace{55mm}(c-2)\hspace{52mm}(c-3) 

\caption{Higher-order Kuramoto model with $K=1$, $\omega = 0$.
(a) Results for $N=3$ with $\delta = \pi/2$.
(a-1) Evolution of the distribution of the phases ${\bm \theta}_3 = (\theta_1, \theta_2, \theta_3)$ plotted on the $(\theta_1-\theta_2) - (\theta_2-\theta_3)$ plane obtained by simulating the higher-order Kuramoto model from $10^5$ random initial conditions sampled with the weights proportional to $\left|\psi^3_{\bm{q}}({\bm \theta}_3)\right|$ (clipped at $10^2$).  Each black point corresponds to a point ${\bm \theta}_3$ in the state space.
(a-2) Stationary solution of the Liouville equation, $\left|\psi^3_{\bm{q}}({\bm \theta}_3)\right|$ (clipped at $10^2$). As before, $\left|\psi^3_{\bm{q}}({\bm \theta}_3)\right|$ diverges on the vertical, horizontal, and diagonal lines, but is plotted in black.
(b) Results for $N=10$ with $\delta=\pi/2$.
Time evolution of the phases ${\bm \theta}_{10}(t) = (\theta_1(t), \theta_2(t), \ldots, \theta_{10}(t))$ (b-1),
$\left|\psi^{10}_{\bm{q}}({\bm \theta}_{10}(t))\right|$
(b-2),
and $\left|\Psi^{10}_{\bm{q}_1,\bm{q}_2}({\bm \theta}_{10}(t))\right|$ (b-3) with time $t$.
(c) Results for $N=10$ with $\delta=0$.
Time evolution of the phases ${\bm \theta}_{10}(t) = (\theta_1(t), \theta_2(t), \ldots, \theta_{10}(t))$ (c-1),
$\left|\psi^{10}_{\bm{q}}({\bm \theta}_{10}(t))\right|$
(c-2),
and $\left|\Psi^{10}_{\bm{q}_1,\bm{q}_2}({\bm \theta}_{10}(t))\right|$ (c-3) with time $t$.
}
\label{fig:higher-order_model}
\end{figure}

\section{Summary}

In this paper, we have exploited a simple relationship between the Perron--Frobenius and Koopman eigenfunctions for a certain class of many-body oscillator systems described by phase models to construct their invariants from a spectral viewpoint.
By taking the ratios of the functions that share the same local growth rate under the Liouvillian dynamics, we can systematically derive Koopman eigenfunctions with zero eigenvalues, namely, the invariants of the system dynamics.
The resulting invariants coincide with those previously found through different methods. Thus, our approach gives a new, alternative framework for understanding the mathematical structure of invariants in coupled-oscillator systems.

Recently, Watanabe--Strogatz theory has been extended to Kuramoto models defined on complex networks, enabling the analysis of invariants in broader and more realistic settings~\cite{cestnik2024integrability}.
Extending our framework to complex networks and systematically investigating the invariants through the Perron--Frobenius and Koopman eigenfunctions constitutes an interesting direction for future research.

\paragraph*{Acknowledgements:}
 K.T. acknowledges JSPS KAKENHI 24K20863 for financial support. H.N. acknowledges JSPS KAKENHI 25H01468, 25K03081, 22K11919, and 22H00516 for financial support.
\paragraph*{Data accessibility:}
This article has no additional data or code.

\clearpage

\appendix

\section{Proof of Theorem~\ref{thm:PF_on_psi}}

We compute each term included in $\mathcal{P}\psi^N_{\bm{q}}$ of Eq.~\eqref{eq:fp11}. 
The first term involving $f$ can be expanded as 
\begin{align}
    &\sum_{j=1}^N \frac{\partial ( f\psi^N_{\bm{q}} )}{\partial \theta_j} = \psi^N_{\bm{q}}\sum_{j=1}^N \bigg[\frac{\partial f}{\partial \theta_j}+ \frac{f}{2}\bigg\{\cot\left(\frac{\theta_{q^{(j-1)}}-\theta_{q^{(j)}}}{2}\right) - \cot\left(\frac{\theta_{q^{(j)}}-\theta_{q^{(j+1)}}}{2}\right)\bigg\}\bigg]\cr
    &= \psi^N_{\bm{q}}\sum_{j=1}^N \frac{\partial f}{\partial \theta_j}.
\end{align}
In the above equation, the term $\{\cdots\}$ vanishes due to the cyclic sum.

Similarly, the second $g \cos \theta_j$ term can be calculated as
\begin{align}
    &\sum_{j=1}^N \frac{\partial g \cos \theta_j \psi_{\bm{q}}^N  }{\partial \theta_j}\cr
    &= \psi^N_{\bm{q}}\sum_{j=1}^N \bigg[\frac{\partial g}{\partial \theta_j}\cos \theta_j+ g\bigg\{- \sin\theta_j+ \frac{1}{2}\cos \theta_{q^{(j)}} \left(\cot\left(\frac{\theta_{q^{(j-1)}}-\theta_{q^{(j)}}}{2}\right) - \cot\left(\frac{\theta_{q^{(j)}}-\theta_{q^{(j+1)}}}{2}\right)\right)\bigg\}\bigg]\cr
    &= \psi^N_{\bm{q}}\sum_{j=1}^N \cos \theta_j \frac{\partial g}{\partial \theta_j},
\end{align}
where the term $\{\cdots\}$ vanishes because of the trigonometric property
\begin{align}
&\sum_{j=1}^N \cos \theta_{q^{(j)}} \left(\cot\left(\frac{\theta_{q^{(j-1)}}-\theta_{q^{(j)}}}{2}\right) - \cot\left(\frac{\theta_{q^{(j)}}-\theta_{q^{(j+1)}}}{2}\right)\right) \cr
&= -\sum_{j=1}^N (\cos\theta_{q^{(j-1)}} -\cos\theta_{q^{(j)}})\cot\left(\frac{\theta_{q^{(j-1)}}-\theta_{q^{(j)}}}{2}\right) \cr
&=\sum_{j=1}^N 2\sin\bigg(\frac{\theta_{q^{(j-1)}}+\theta_{q^{(j)}}}{2}\bigg)\sin\bigg(\frac{\theta_{q^{(j-1)}}-\theta_{q^{(j)}}}{2}\bigg)\cot\left(\frac{\theta_{q^{(j-1)}}-\theta_{q^{(j)}}}{2}\right) \cr
&=\sum_{j=1}^N 2\sin\bigg(\frac{\theta_{q^{(j-1)}}+\theta_{q^{(j)}}}{2}\bigg)\cos\bigg(\frac{\theta_{q^{(j-1)}}-\theta_{q^{(j)}}}{2}\bigg) \cr
&= \sum_{j=1}^N [ \sin \theta_{q^{(j-1)}} + \sin \theta_{q^{(j)}}]
\end{align}
and the cyclic sum,
\begin{align}
    \sum_{j=1}^N \left[-\sin\theta_j + \frac{1}{2} \left(\sin \theta_{q^{(j-1)}} + \sin \theta_{q^{(j)}}\right)\right] = 0.
\end{align}

For the last $h \sin \theta_j$ term, we obtain
\begin{align}
    &\sum_{j=1}^N \frac{\partial  h\sin \theta_j \psi^N_{\bm{q}} }{\partial \theta_j}\cr
    &= \psi^N_{\bm{q}}\sum_{j=1}^N\bigg[\frac{\partial h}{\partial \theta_j}\sin \theta_j
    + h\bigg\{\cos \theta_j + \frac{1}{2}\sin \theta_{q^{(j)}} \left(\cot\left(\frac{\theta_{q^{(j-1)}}-\theta_{q^{(j)}}}{2}\right) - \cot\left(\frac{\theta_{q^
    {(j)}}-\theta_{q^{(j+1)}}}{2}\right)\right)\bigg\}\bigg]\cr
    &= \psi^N_{\bm{q}}\sum_{j=1}^N \frac{\partial h}{\partial \theta_j}\sin \theta_j.
\end{align}
Again, the term $\{\cdots\}$ vanishes because of the trigonometric property
\begin{align}
&\sum_{j=1}^N \sin \theta_{q^{(j)}} \left(\cot\left(\frac{\theta_{q^{(j-1)}}-\theta_{q^{(j)}}}{2}\right) - \cot\left(\frac{\theta_{q^{(j)}}-\theta_{q^{(j+1)}}}{2}\right)\right) \cr
&= -\sum_{j=1}^N \left(\sin\theta_{q^{(j-1)}} -\sin\theta_{q^{(j)}}\right) \cot\left(\frac{\theta_{q^{(j-1)}}-\theta_{q^{(j)}}}{2}\right) \cr
&=-\sum_{j=1}^N 2\cos\bigg(\frac{\theta_{q^{(j-1)}}+\theta_{q^{(j)}}}{2}\bigg)\sin\bigg(\frac{\theta_{q^{(j-1)}}-\theta_{q^{(j)}}}{2}\bigg)\cot\left(\frac{\theta_{q^{(j-1)}}-\theta_{q^{(j)}}}{2}\right) \cr
&=-\sum_{j=1}^N 2\cos\bigg(\frac{\theta_{q^{(j-1)}}+\theta_{q^{(j)}}}{2}\bigg)\cos\bigg(\frac{\theta_{q^{(j-1)}}-\theta_{q^{(j)}}}{2}\bigg) \cr
&=-\sum_{j=1}^N \left[ \cos \theta_{q^{(j-1)}} + \cos \theta_{q^{(j)}}\right]
\end{align}
and the cyclic sum,
\begin{align}
    \sum_{j=1}^N \left[\cos\theta_j - \frac{1}{2} \left(\cos \theta_{q^{(j-1)}} + \cos \theta_{q^{(j)}}\right)\right] = 0.
\end{align}

Combining all the terms, we obtain
\begin{align}
    \mathcal{P} \psi_{\bm{q}}^N(\bm{\theta}_N) = \Lambda(\bm{\theta}_N, t) \psi_{\bm{q}}^N(\bm{\theta}_N),
\end{align}
where
\begin{align}
    \Lambda(\bm{\theta}_N, t)
    := -\sum_{j=1}^N \left(
    \frac{\partial f(\bm{\theta}_N, t)}{\partial \theta_j}
    + \cos \theta_j \frac{\partial g(\bm{\theta}_N, t)}{\partial \theta_j}
    + \sin \theta_j \frac{\partial h(\bm{\theta}_N, t)}{\partial \theta_j}
    \right).
\end{align}

\section{Independent Invariants}
\label{sec:independent_invariants}

\subsection{Riccati equation and cross ratios}

We here discuss the relation of the invariants with the Riccati equation.
While the $N-3$ invariants of Eq.~\eqref{eq:time_dependent_equation} or~\eqref{eq:state_dependent_equation} were originally found by the Watanabe--Strogatz transformation, they can also be obtained from the property of the Riccati equation into which Eq.~\eqref{eq:time_dependent_equation} can be transformed ~\cite{Watanabe1994constants,Marvel2009identical}.
 
The Riccati equation~\cite{reid1972riccati} is a nonlinear ordinary differential equation defined as
\begin{align}
    \label{eq:riccati_equation}
    \frac{dx}{dt} = \alpha(t)x + \beta(t)x^2 + \gamma(t),
\end{align}
where $\alpha(t)$, $\beta(t)$, and $\gamma(t)$ are time-dependent functions. Notably, given four distinct solutions to the Riccati equation, one can construct an invariant quantity termed \textit{cross ratio}~\cite{reid1972riccati}, expressed as follows:
\begin{align}
    \label{eq:cross_ratio}
    \frac{(x_1 - x_2)(x_3 - x_4)}{(x_2 - x_3)(x_4 - x_1)}.
\end{align}
Here, $x_1(t)$, $x_2(t)$, $x_3(t)$, and $x_4(t)$ are the solutions of Eq.~\eqref{eq:riccati_equation} with distinct initial conditions $x_1(0)$, $x_2(0)$, $x_3(0)$, and $x_4(0)$.

Equation~\eqref{eq:time_dependent_equation} is transformed into the form of a Riccati equation by considering the time evolution of the complex variable $x = e^{i\theta}$~\cite{Watanabe1994constants,Marvel2009identical}:
\begin{align}
    \frac{dx}{dt} = \frac{d}{dt}e^{i\theta} &= i\left(f(t) + g(t)\cos \theta + h(t)\sin \theta \right)e^{i\theta} \cr
    &=if(t)x + \frac{i}{2}(g(t)+ih(t))x^2 + \frac{i}{2}(g(t)-ih(t)).
\end{align}
Therefore, Eq.~\eqref{eq:time_dependent_equation} corresponds to the Riccati equation~\eqref{eq:riccati_equation} with the following coefficients:
\begin{align}
    \alpha(t) &= i f(t),\cr
    \beta(t) &= \frac{i}{2}\left(g(t) + i h(t)\right),\cr
    \gamma(t) &= \frac{i}{2}\left(g(t) - i h(t)\right).
\end{align}
It follows that, for four distinct solutions $\theta_1$, $\theta_2$, $\theta_3$, and $\theta_4$ of Eq.~\eqref{eq:time_dependent_equation}, the following quantity:
\begin{align}
    \frac{(e^{i\theta_1} - e^{i\theta_2})(e^{i\theta_3} - e^{i\theta_4})}{(e^{i\theta_2} - e^{i\theta_3})(e^{i\theta_4} - e^{i\theta_1})}
\end{align}
is an invariant. The cross ratio in Eq.~\eqref{eq:cross_ratio} can be equivalently rewritten in terms of sine functions as~(see \cite{Marvel2009identical})
\begin{align}
\label{eq:bracket}
    \langle 1,2,3,4\rangle = \frac{\sin\left(\frac{\theta_1 - \theta_2}{2}\right) \sin\left(\frac{\theta_3 - \theta_4}{2}\right)}{\sin\left(\frac{\theta_2 - \theta_3}{2}\right) \sin\left(\frac{\theta_4 - \theta_1}{2}\right)},
\end{align}
where we introduced the notation $\langle i,j,k,l\rangle$ for the cross ratio.
From $N$ distinct solutions of Eq.~\eqref{eq:time_dependent_equation}, one can construct $N(N-1)(N-2)(N-3)$ (the total number of ordered four indices) such invariants of the form $\langle i,j,k,l\rangle$, but only $N - 3$ of them are independent functions~\cite{Marvel2009identical}. 

\subsection{Perron--Frobenius eigenfunctions and an invariant for $N=4$}
\label{sec:independent}

Before we discuss the relation between the WS invariants and $\Psi^N_{\bm{q}_1, \bm{q}_2}$, we show that only one functionally independent invariant given in the form $\Psi^4_{\bm{q}_1,\bm{q}_2}$ exists when $N=4$. 
The function $\psi^4_{\bm{q}}$ satisfies the following cyclic symmetry:

\begin{align}
    \psi^4_{(q^{(1)}, q^{(2)}, q^{(3)}, q^{(4)})} 
    = \psi^4_{(q^{(2)}, q^{(3)}, q^{(4)}, q^{(1)})} 
    = \psi^4_{(q^{(3)}, q^{(4)}, q^{(1)}, q^{(2)})} 
    = \psi^4_{(q^{(4)}, q^{(1)}, q^{(2)}, q^{(3)})},
\end{align}
and reversal symmetry:
\begin{align}
    \psi^4_{(q^{(1)}, q^{(2)}, q^{(3)}, q^{(4)})}
    = \psi^4_{(q^{(4)}, q^{(3)}, q^{(2)}, q^{(1)})}.
\end{align}
Given these symmetries, only three functions remain:
\begin{align}
    \label{eq:ApxA_psi4}
    \psi^4_{(1,2,3,4)}, \quad
    \psi^4_{(1,3,4,2)}, \quad
    \psi^4_{(1,4,2,3)}.
\end{align}
In addition, one more relation exists:
\begin{align}
    \psi^4_{(1,2,3,4)} + \psi^4_{(1,3,4,2)} + \psi^4_{(1,4,2,3)} = 0,
\end{align}
namely, only two of the functions in Eq.~\eqref{eq:ApxA_psi4} are independent, such as
\begin{align}
        \psi^4_{(1,2,3,4)}, \quad  \psi^4_{(1,3,4,2)}.
\end{align}
From these functions, we can construct only one independent invariant $\Psi^4_{(1,2,3,4), (1,3,4,2)}$. 

\subsection{\texorpdfstring
  {Relationship between $\Psi^N_{\bm{q}}$ and $\langle i,j,k,l\rangle$}
  {Relationship between Psi^N_q and <i,j,k,l>}
}

Here, we show that the set $\bm{\Psi}^N = \{\Psi^N_{\bm{q}_1, \bm{q}_2}\}_{\bm{q}_1, \bm{q}_2}$ includes all invariants of the form $\langle q^{(1)}, q^{(2)}, q^{(3)}, q^{(4)}\rangle$, introduced in Eq.~\eqref{eq:bracket},
and that each function $\Psi^N_{\bm{q}_1, \bm{q}_2}$ can be decomposed into a product of such terms. Consequently, the number of independent invariants in $\bm{\Psi}^N$ is $N - 3$ (except for a constant function, which is a trivial invariant).

Our discussion is as follows. We compare invariants in $\bm{\Psi}^N$ with $\langle i,j,k,l\rangle$, and, using the fact that the number of independent functions of $\langle i,j,k,l\rangle$ is $N-3$~\cite{Marvel2009identical}, we show that the number of functions in $\bm{\Psi}^N$ is also $N-3$.  
To this end, we show that invariants in $\bm{\Psi}^N$ can be constructed as products of $\langle i,j,k,l\rangle$. 
First, we show that $\bm{\Psi}^4$ coincides with $\langle i,j,k,l\rangle$. Then, for $N \geq 5$, we show that $\bm{\Psi}^N$ can be expressed as $\bm{\Psi}^{N-1}$ multiplied by several terms of $\langle i,j,k,l\rangle$.  
It follows that any $\bm{\Psi}^N$ can be expressed as a product of $\langle i,j,k,l\rangle$, implying that the number of independent $\bm{\Psi}^N$ is at most $N-3$.  
Next, we show that the set of $\bm{\Psi}^N$ includes $\langle i,j,k,l\rangle$. This implies that the number of independent $\bm{\Psi}^N$ is at least $N-3$.  
Combining these results, we conclude that the number of independent $\bm{\Psi}^N$ is exactly $N-3$.

We utilize the symmetric properties of the cross ratio $\langle i,j,k,l\rangle$~\cite{Marvel2009identical} that any permutation of the four indices yields an expression functionally dependent on $\langle i,j,k,l\rangle$. That is, all such variants can be written as rational functions of a single cross ratio and thus they are not independent.

We begin with the case $N = 4$. 
As we discussed in Appendix~\ref{sec:independent}, only one independent invariant can be obtained as $\Psi^4_{\bm{q}_1,\bm{q}_2}$ in this case. Let us choose $\Psi^4_{(1,3,4,2), (1,4,2,3)}$ as the representative one. 
By calculating the fraction, we obtain the following relation:
\begin{align}
    \label{eq:Psi4}
    \Psi^4_{(1,3,4,2), (1,4,2,3)} &= \frac{\psi^4_{(1,4,2,3)}}{\psi^4_{(1,3,4,2)}} \cr
    &=\frac{\sin\left( \frac{\theta_{1} - \theta_{3}}{2} \right)\sin\left( \frac{\theta_{3} - \theta_{4}}{2} \right)\sin\left( \frac{\theta_{4} - \theta_{2}}{2} \right)\sin\left( \frac{\theta_{2} - \theta_{1}}{2} \right)}{\sin\left( \frac{\theta_{1} - \theta_{4}}{2} \right)\sin\left( \frac{\theta_{4} - \theta_{2}}{2} \right)\sin\left( \frac{\theta_{2} - \theta_{3}}{2} \right)\sin\left( \frac{\theta_{3} - \theta_{1}}{2} \right)}\cr
    &=-\frac{ \sin\left( \frac{\theta_{1} - \theta_{2}}{2} \right)\sin\left( \frac{\theta_{3} - \theta_{4}}{2} \right)
   }{\sin\left( \frac{\theta_{2} - \theta_{3}}{2} \right)\sin\left( \frac{\theta_{4} - \theta_{1}}{2} \right)}\cr
   &=-\langle 1,2,3,4\rangle.
\end{align}

Though not independent of Eq.~\eqref{eq:Psi4}, we note that $\Psi^4_{(1,4,2,3), (1,2,3,4)}$ and $\Psi^4_{(1,2,3,4), (1,3,4,2)}$ can also be constructed as the ratio of $\psi^4_{\bm{q}}$ and $\psi^4_{\bm{q}'}$ with different permutation $\bm{q}$ and $\bm{q}'$, and we can find the following relations:
\begin{align}
    \Psi^4_{(1,4,2,3), (1,2,3,4)} &= -\langle 1,3,4,2\rangle, \\
    \Psi^4_{(1,2,3,4), (1,3,4,2)} &= -\langle 1,4,2,3\rangle.
\end{align}

Then, let us consider the case $N\geq 5$. 
\begin{lemma}
\label{lem:decomposition}
    Let $N \geq 5$, and let $\Psi^N_{\bm{q}_1, \bm{q}_2}$ be defined as before. Then $\Psi^N_{\bm{q}_1, \bm{q}_2}$ can be reduced to $\Psi^{N-1}_{\bm{q}'_1, \bm{q}'_2}$ or decomposed as a product of $\Psi^{N-1}_{\bm{q}'_1, \bm{q}'_2}$ and invariants $\langle  i,j,k,l\rangle$.
    Here,    
    $\bm{q}'_1$, $\bm{q}'_2$ are the permutations of $\{1,2,\ldots,N-1\}$ obtained by removing $N$ 
    from $\bm{q}_1$ and $\bm{q}_2$:
    \begin{align}
        q_1'^{(l)} &=
        \begin{cases}
            q_1^{(l)}, & \text{if } l < \xi, \\
            q_1^{(l+1)}, & \text{if } l \geq \xi,
        \end{cases} \quad
        q_2'^{(l)} =
        \begin{cases}
            q_2^{(l)}, & \text{if } l < \zeta, \\
            q_2^{(l+1)}, & \text{if } l \geq \zeta,
        \end{cases}
    \end{align}
    where $q_1^{(\xi)} = q_2^{(\zeta)} = N$, and we impose the cyclic property in the indices as $q_{1,2}^{(0)}=q_{1,2}^{(N)}$ and $q_{1,2}^{(N+1)}=q_{1,2}^{(1)}$.
\end{lemma}

\begin{proof}
From Eq.~\eqref{eq:eigenfunction}, we can decompose $\Psi^N_{\bm{q}_1, \bm{q}_2}$ as
\begin{align}
    \Psi^N_{\bm{q}_1, \bm{q}_2}
    = \Psi^{N-1}_{\bm{q}'_1, \bm{q}'_2}\frac{\sin\left( \frac{\theta_{q_1^{(\xi-1)}} - \theta_{q_1^{(\xi)}}}{2} \right)\sin\left( \frac{\theta_{q_1^{(\xi)}} - \theta_{q_1^{(\xi+1)}}}{2} \right)}
           {\sin\left( \frac{\theta_{q_2^{(\zeta-1)}} - \theta_{q_2^{(\zeta)}}}{2} \right)\sin\left( \frac{\theta_{q_2^{(\zeta)}} - \theta_{q_2^{(\zeta+1)}}}{2} \right)}\cdot \frac{\sin\left( \frac{\theta_{q_2^{(\zeta-1)}} - \theta_{q_2^{(\zeta+1)}}}{2} \right)}{\sin\left( \frac{\theta_{q_1^{(\xi-1)}} - \theta_{q_1^{(\xi+1)}}}{2} \right)}.
\end{align}
We analyze this by dividing into the following six cases. 
\begin{enumerate}
\item[\textit{Case.1:}] $\{{q_1^{(\xi-1)}},{q_1^{(\xi+1)}}\} = \{{q_2^{(\zeta-1)}}, {q_2^{(\zeta+1)}}\}$\\ \\
In this case, the following equation holds: 
\begin{align}
    \frac{\sin\left( \frac{\theta_{q_1^{(\xi-1)}} - \theta_{q_1^{(\xi)}}}{2} \right)\sin\left( \frac{\theta_{q_1^{(\xi)}} - \theta_{q_1^{(\xi+1)}}}{2} \right)}
           {\sin\left( \frac{\theta_{q_2^{(\zeta-1)}} - \theta_{q_2^{(\zeta)}}}{2} \right)\sin\left( \frac{\theta_{q_2^{(\zeta)}} - \theta_{q_2^{(\zeta+1)}}}{2} \right)} = 1.
\end{align}
Thus, we obtain
\begin{align}
    \Psi^N_{\bm{q}_1, \bm{q}_2} = \Psi^{N-1}_{\bm{q}'_1, \bm{q}'_2}.
\end{align}

This includes both the aligned case $(q_1^{(\xi-1)},q_1^{(\xi+1)})=(q_2^{(\zeta-1)},q_2^{(\zeta+1)})$ and the swapped case $(q_1^{(\xi-1)},q_1^{(\xi+1)})=(q_2^{(\zeta+1)},q_2^{(\zeta-1)})$, and in either situation the sine factors cancel pairwise, yielding the ratio $1$.

\item[\textit{Case.2:}] $q_1^{(\xi-1)} = q_2^{(\zeta-1)} = M$, ${q_1^{(\xi+1)}} \ne {q_2^{(\zeta+1)}}$ \\ \\
In this case, the following relation holds:
\begin{align}
    &\frac{\sin\left( \frac{\theta_{q_1^{(\xi-1)}} - \theta_{q_1^{(\xi)}}}{2} \right)\sin\left( \frac{\theta_{q_1^{(\xi)}} - \theta_{q_1^{(\xi+1)}}}{2} \right)\sin\left( \frac{\theta_{q_2^{(\zeta-1)}} - \theta_{q_2^{(\zeta+1)}}}{2} \right)}
   {\sin\left( \frac{\theta_{q_2^{(\zeta-1)}} - \theta_{q_2^{(\zeta)}}}{2} \right)\sin\left( \frac{\theta_{q_2^{(\zeta)}} - \theta_{q_2^{(\zeta+1)}}}{2} \right)\sin\left( \frac{\theta_{q_1^{(\xi-1)}} - \theta_{q_1^{(\xi+1)}}}{2} \right)}\cr
    &=\frac{\sin\left( \frac{\theta_{q_1^{(\xi)}} - \theta_{q_1^{(\xi+1)}}}{2} \right)\sin\left( \frac{\theta_{q_2^{(\zeta-1)}} - \theta_{q_2^{(\zeta+1)}}}{2} \right)}
   {\sin\left( \frac{\theta_{q_2^{(\zeta)}} - \theta_{q_2^{(\zeta+1)}}}{2} \right)\sin\left( \frac{\theta_{q_1^{(\xi-1)}} - \theta_{q_1^{(\xi+1)}}}{2} \right)}\cr
   &=\frac{\sin\left( \frac{\theta_{N} - \theta_{q_1^{(\xi+1)}}}{2} \right)\sin\left( \frac{\theta_{M} - \theta_{q_2^{(\zeta+1)}}}{2} \right)}
   {\sin\left( \frac{\theta_{N} - \theta_{q_2^{(\zeta+1)}}}{2} \right)\sin\left( \frac{\theta_{M} - \theta_{q_1^{(\xi+1)}}}{2} \right)}\cr
   &= \langle N,q_1^{(\xi+1)},M, q_2^{(\zeta+1)}\rangle.
\end{align}
Thus, we obtain
\begin{align}
    \Psi^N_{\bm{q}_1, \bm{q}_2} = \Psi^{N-1}_{\bm{q}'_1, \bm{q}'_2}\langle N,q_1^{(\xi+1)},M, q_2^{(\zeta+1)}\rangle.
\end{align}

\item[\textit{Case.3:}] ${q_1^{(\xi+1)}} = {q_2^{(\zeta+1)}} = M$, ${q_1^{(\xi-1)}} \ne {q_2^{(\zeta-1)}}$\\ \\
In this case, the following relation holds:
\begin{align}
    &\frac{\sin\left( \frac{\theta_{q_1^{(\xi-1)}} - \theta_{q_1^{(\xi)}}}{2} \right)\sin\left( \frac{\theta_{q_1^{(\xi)}} - \theta_{q_1^{(\xi+1)}}}{2} \right)\sin\left( \frac{\theta_{q_2^{(\zeta-1)}} - \theta_{q_2^{(\zeta+1)}}}{2} \right)}
   {\sin\left( \frac{\theta_{q_2^{(\zeta-1)}} - \theta_{q_2^{(\zeta)}}}{2} \right)\sin\left( \frac{\theta_{q_2^{(\zeta)}} - \theta_{q_2^{(\zeta+1)}}}{2} \right)\sin\left( \frac{\theta_{q_1^{(\xi-1)}} - \theta_{q_1^{(\xi+1)}}}{2} \right)}\cr
    &=\frac{\sin\left( \frac{\theta_{q_1^{(\xi-1)}} - \theta_{q_1^{(\xi)}}}{2} \right)\sin\left( \frac{\theta_{q_2^{(\zeta-1)}} - \theta_{q_2^{(\zeta+1)}}}{2} \right)}
   {\sin\left( \frac{\theta_{q_2^{(\zeta-1)}} - \theta_{q_2^{(\zeta)}}}{2} \right)\sin\left( \frac{\theta_{q_1^{(\xi-1)}} - \theta_{q_1^{(\xi+1)}}}{2} \right)}\cr
    &=\frac{\sin\left( \frac{\theta_{q_1^{(\xi-1)}} - \theta_{N}}{2} \right)\sin\left( \frac{\theta_{q_2^{(\zeta-1)}} - \theta_{M}}{2} \right)}
   {\sin\left( \frac{\theta_{q_2^{(\zeta-1)}} - \theta_{N}}{2} \right)\sin\left( \frac{\theta_{q_1^{(\xi-1)}} - \theta_{M}}{2} \right)}\cr
    &= \langle N,q_1^{(\xi-1)},M, q_2^{(\zeta-1)}\rangle.
\end{align}
Thus, we obtain
\begin{align}
    \Psi^N_{\bm{q}_1, \bm{q}_2} = \Psi^{N-1}_{\bm{q}'_1, \bm{q}'_2}\langle N,q_1^{(\xi-1)},M, q_2^{(\zeta-1)}\rangle.
\end{align}

\item[\textit{Case.4:}] ${q_1^{(\xi-1)}} = {q_2^{(\zeta+1)}} = M$, ${q_1^{(\xi+1)}} \ne {q_2^{(\zeta-1)}}$\\ \\
In this case, the following relation holds:
\begin{align}
    &\frac{\sin\left( \frac{\theta_{q_1^{(\xi-1)}} - \theta_{q_1^{(\xi)}}}{2} \right)\sin\left( \frac{\theta_{q_1^{(\xi)}} - \theta_{q_1^{(\xi+1)}}}{2} \right)\sin\left( \frac{\theta_{q_2^{(\zeta-1)}} - \theta_{q_2^{(\zeta+1)}}}{2} \right)}
   {\sin\left( \frac{\theta_{q_2^{(\zeta-1)}} - \theta_{q_2^{(\zeta)}}}{2} \right)\sin\left( \frac{\theta_{q_2^{(\zeta)}} - \theta_{q_2^{(\zeta+1)}}}{2} \right)\sin\left( \frac{\theta_{q_1^{(\xi-1)}} - \theta_{q_1^{(\xi+1)}}}{2} \right)}\cr
    &=-\frac{\sin\left( \frac{\theta_{q_1^{(\xi)}} - \theta_{q_1^{(\xi+1)}}}{2} \right)\sin\left( \frac{\theta_{q_2^{(\zeta-1)}} - \theta_{q_2^{(\zeta+1)}}}{2} \right)}
   {\sin\left( \frac{\theta_{q_2^{(\zeta-1)}} - \theta_{q_2^{(\zeta)}}}{2} \right)\sin\left( \frac{\theta_{q_1^{(\xi-1)}} - \theta_{q_1^{(\xi+1)}}}{2} \right)}\cr
    &=-\frac{\sin\left( \frac{\theta_{N} - \theta_{q_1^{(\xi+1)}}}{2} \right)\sin\left( \frac{\theta_{q_2^{(\zeta-1)}} - \theta_{M}}{2} \right)}
   {\sin\left( \frac{\theta_{q_2^{(\zeta-1)}} - \theta_{N}}{2} \right)\sin\left( \frac{\theta_{M} - \theta_{q_1^{(\xi+1)}}}{2} \right)}\cr
    &= -\langle N,q_1^{(\xi+1)},M, q_2^{(\zeta-1)}\rangle.
\end{align}
Thus, we obtain
\begin{align}
    \Psi^N_{\bm{q}_1, \bm{q}_2} = -\Psi^{N-1}_{\bm{q}'_1, \bm{q}'_2}\langle N,q_1^{(\xi+1)},M, q_2^{(\zeta-1)}\rangle.
\end{align}

\item[\textit{Case.5:}] ${q_1^{(\xi+1)}} = {q_2^{(\zeta-1)}} = M$, ${q_1^{(\xi-1)}} \ne {q_2^{(\zeta+1)}}$\\ \\
In this case, the following relation holds:
\begin{align}
    &\frac{\sin\left( \frac{\theta_{q_1^{(\xi-1)}} - \theta_{q_1^{(\xi)}}}{2} \right)\sin\left( \frac{\theta_{q_1^{(\xi)}} - \theta_{q_1^{(\xi+1)}}}{2} \right)\sin\left( \frac{\theta_{q_2^{(\zeta-1)}} - \theta_{q_2^{(\zeta+1)}}}{2} \right)}
   {\sin\left( \frac{\theta_{q_2^{(\zeta-1)}} - \theta_{q_2^{(\zeta)}}}{2} \right)\sin\left( \frac{\theta_{q_2^{(\zeta)}} - \theta_{q_2^{(\zeta+1)}}}{2} \right)\sin\left( \frac{\theta_{q_1^{(\xi-1)}} - \theta_{q_1^{(\xi+1)}}}{2} \right)}\cr
    &=-\frac{\sin\left( \frac{\theta_{q_1^{(\xi-1)}} - \theta_{q_1^{(\xi)}}}{2} \right)\sin\left( \frac{\theta_{q_2^{(\zeta-1)}} - \theta_{q_2^{(\zeta+1)}}}{2} \right)}
   {\sin\left( \frac{\theta_{q_2^{(\zeta)}} - \theta_{q_2^{(\zeta+1)}}}{2} \right)\sin\left( \frac{\theta_{q_1^{(\xi-1)}} - \theta_{q_1^{(\xi+1)}}}{2} \right)}\cr
    &=-\frac{\sin\left( \frac{\theta_{q_1^{(\xi-1)}} - \theta_{N}}{2} \right)\sin\left( \frac{\theta_{M} - \theta_{q_2^{(\zeta+1)}}}{2} \right)}
   {\sin\left( \frac{\theta_{N} - \theta_{q_2^{(\zeta+1)}}}{2} \right)\sin\left( \frac{\theta_{q_1^{(\xi-1)}} - \theta_{M}}{2} \right)}\cr
    &= -\langle N,q_1^{(\xi-1)},M, q_2^{(\zeta+1)}\rangle.
\end{align}
Thus, we obtain
\begin{align}
    \Psi^N_{\bm{q}_1, \bm{q}_2} = -\Psi^{N-1}_{\bm{q}'_1, \bm{q}'_2}\langle N,q_1^{(\xi-1)},M, q_2^{(\zeta+1)}\rangle.
\end{align}

\item[\textit{Case.6:}] $\{{q_1^{(\xi-1)}},{q_1^{(\xi+1)}}\} \cap \{{q_2^{(\zeta-1)}}, {q_2^{(\zeta+1)}}\} = \varnothing$ 
\\ \\
In this case, the following relation holds:
\begin{align}
    &\frac{\sin\left( \frac{\theta_{q_1^{(\xi-1)}} - \theta_{q_1^{(\xi)}}}{2} \right)\sin\left( \frac{\theta_{q_1^{(\xi)}} - \theta_{q_1^{(\xi+1)}}}{2} \right)\sin\left( \frac{\theta_{q_2^{(\zeta-1)}} - \theta_{q_2^{(\zeta+1)}}}{2} \right)}
   {\sin\left( \frac{\theta_{q_2^{(\zeta-1)}} - \theta_{q_2^{(\zeta)}}}{2} \right)\sin\left( \frac{\theta_{q_2^{(\zeta)}} - \theta_{q_2^{(\zeta+1)}}}{2} \right)\sin\left( \frac{\theta_{q_1^{(\xi-1)}} - \theta_{q_1^{(\xi+1)}}}{2} \right)}\cr
    &=\frac{\sin\left( \frac{\theta_{q_1^{(\xi-1)}} - \theta_{q_1^{(\xi)}}}{2} \right)\sin\left( \frac{\theta_{q_1^{(\xi)}} - \theta_{q_1^{(\xi+1)}}}{2} \right)\sin\left( \frac{\theta_{q_2^{(\zeta-1)}} - \theta_{q_2^{(\zeta+1)}}}{2} \right)}
   {\sin\left( \frac{\theta_{q_2^{(\zeta-1)}} - \theta_{q_2^{(\zeta)}}}{2} \right)\sin\left( \frac{\theta_{q_2^{(\zeta)}} - \theta_{q_2^{(\zeta+1)}}}{2} \right)\sin\left( \frac{\theta_{q_1^{(\xi-1)}} - \theta_{q_1^{(\xi+1)}}}{2} \right)}
   \cdot \frac{\sin\left( \frac{\theta_{q_1^{(\xi-1)}} - \theta_{q_2^{(\zeta-1)}}}{2} \right)}{\sin\left( \frac{\theta_{q_1^{(\xi-1)}} - \theta_{q_2^{(\zeta-1)}}}{2} \right)}\cr
    &=\frac{\sin\left( \frac{\theta_{N} - \theta_{q_1^{(\xi+1)}}}{2} \right)\sin\left( \frac{\theta_{q_1^{(\xi-1)}} - \theta_{q_2^{(\zeta-1)}}}{2} \right)}{\sin\left( \frac{\theta_{q_1^{(\xi-1)}} - \theta_{q_1^{(\xi+1)}}}{2} \right)\sin\left( \frac{\theta_{q_2^{(\zeta-1)}} - \theta_{N}}{2} \right)}\cdot
    \frac{\sin\left( \frac{\theta_{q_1^{(\xi-1)}} - \theta_{N}}{2} \right)\sin\left( \frac{\theta_{q_2^{(\zeta-1)}} - \theta_{q_2^{(\zeta+1)}}}{2} \right)}{\sin\left( \frac{\theta_{q_1^{(\xi-1)}} - \theta_{q_2^{(\zeta-1)}}}{2} \right)\sin\left( \frac{\theta_{N} - \theta_{q_2^{(\zeta+1)}}}{2} \right)}
    \cr
    &=-\langle N, q_1^{(\xi+1)}, q_1^{(\xi-1)}, q_2^{(\zeta-1)}\rangle\langle N, q_1^{(\xi-1)}, q_2^{(\zeta-1)}, q_2^{(\zeta+1)}\rangle.
\end{align}
Thus, we obtain
\begin{align}
    \Psi^N_{\bm{q}_1, \bm{q}_2}= -\Psi^{N-1}_{\bm{q}'_1, \bm{q}'_2}
    \langle N, q_1^{(\xi+1)}, q_1^{(\xi-1)}, q_2^{(\zeta-1)}\rangle\langle N, q_1^{(\xi-1)}, q_2^{(\zeta-1)}, q_2^{(\zeta+1)}\rangle.
\end{align}
\end{enumerate}
From the above results, which cover all possible cases,  $\Psi^N_{\bm{q}_1, \bm{q}_2}$ is reduced to $\Psi^{N-1}_{\bm{q}'_1, \bm{q}'_2}$ or decomposed as a product of $\Psi^{N-1}_{\bm{q}'_1, \bm{q}'_2}$ and invariants $\langle  i,j,k,l\rangle$.
\end{proof}

Using Lemma~\ref{lem:decomposition}, we obtain the following Theorem by mathematical induction.
\begin{theorem}
\label{thm:atmost}
    For all $N \geq 4$, the invariant $\Psi^N_{\bm{q}_1, \bm{q}_2}$ can be decomposed into a product of  
    invariants of the form $\langle i,j,k,l\rangle$. 
\end{theorem}

We have shown that each $\Psi^N_{\bm{q}_1, \bm{q}_2}$ can be expressed as a product of cross-ratio invariants $\langle i,j,k,l\rangle$. Since the number of independent cross ratios for $N$ oscillators is known to be at most $N - 3$, it follows that the number of independent functions in $\bm{\Psi}^N = \{\Psi^N_{\bm{q}_1, \bm{q}_2}\}_{\bm{q}_1, \bm{q}_2}$ is at most $N - 3$.

To complete the argument, we now show that the set of all cross-ratio invariants $\{\langle i,j,k,l\rangle\}_{i,j,k,l}$ is contained within the set $\bm{\Psi}^N$.

\begin{theorem}
    \label{thm:atleast}
    For any $N \geq 4$ and any distinct indices $i, j, k, l \in \{1, 2, \ldots, N\}$, the invariant $\langle i,j,k,l\rangle$ can be expressed in the form of $\Psi^N_{\bm{q}_1, \bm{q}_2}$.
\end{theorem}

\begin{proof}
For $N = 4$, it has already been discussed that $\Psi^4_{\bm{q}_1,\bm{q}_2}$ is equivalent to the cross ratio. 
We now consider the cases with $N \geq 5$. Let $\bm{q} = (\tilde{\bm{q}},i,j,k,l)$ be a permutation of $\{1,2,\ldots,N\}$, where $\tilde{\bm{q}} = (\tilde{q}^{(1)},\tilde{q}^{(2)},\ldots,\tilde{q}^{(N-4)})$ is a list of the remaining $N - 4$ indices.
    By simplifying the function $\Psi^N_{(\tilde{\bm{q}}, i,j,l,k),(\tilde{\bm{q}},i,l,j,k)}$, 
    we obtain the following relation:
    \begin{align}
        &\Psi^N_{(\tilde{\bm{q}},i,j,l,k),(\tilde{\bm{q}},i,l,j,k)}
        = \frac{\psi^N_{(\tilde{\bm{q}},i,l,j,k)}}{\psi^N_{(\tilde{\bm{q}},i,j,l,k)}}\cr
        &=\frac{\prod_{m = 1}^{N-5}\sin\left( \frac{\theta_{\tilde{q}^{(m)}} - \theta_{\tilde{q}^{(m+1)}}}{2} \right)\sin\left( \frac{\theta_{\tilde{q}^{(N-4)}} - \theta_{i}}{2} \right)\sin\left(\frac{\theta_{i}-\theta_{j}}{2}\right)\sin\left(\frac{\theta_{j}-\theta_{l}}{2}\right)\sin\left(\frac{\theta_{l}-\theta_{k}}{2}\right)\sin\left(\frac{\theta_{k}-\theta_{\tilde{q}^{(1)}}}{2}\right)}
        {\prod_{m = 1}^{N-5}\sin\left( \frac{\theta_{\tilde{q}^{(m)}} - \theta_{\tilde{q}^{(m+1)}}}{2} \right)\sin\left( \frac{\theta_{\tilde{q}^{(N-4)}} - \theta_{i}}{2} \right)\sin\left(\frac{\theta_{i}-\theta_{l}}{2}\right)\sin\left(\frac{\theta_{l}-\theta_{j}}{2}\right)\sin\left(\frac{\theta_{j}-\theta_{k}}{2}\right)\sin\left(\frac{\theta_{k}-\theta_{\tilde{q}^{(1)}}}{2}\right)}\cr
        &=-\frac{\sin\left(\frac{\theta_{i}-\theta_{j}}{2}\right)\sin\left(\frac{\theta_{k}-\theta_{l}}{2}\right)}
        {\sin\left(\frac{\theta_{j}-\theta_{k}}{2}\right)\sin\left(\frac{\theta_{l}-\theta_{i}}{2}\right)}
        = -\langle i,j,k,l\rangle.
    \end{align}
    Therefore, any $\langle i,j,k,l\rangle$ can be written in the form of $\Psi^N_{\bm{q}_1,\bm{q}_2}$ for $N \geq 4$.
\end{proof}

Since $\bm{\Psi}^N$ contains all the invariants of the form $\langle i,j,k,l\rangle$ and the number of independent $\langle i,j,k,l\rangle$ is $N-3$~\cite{Marvel2009identical}, the number of independent functions in $\bm{\Psi}^N$ is at least $N-3$. 
Thus, we conclude that the number of independent elements in $\bm{\Psi}^N$ is exactly $N-3$. 

\section{Koopman Eigenfunctions of Theta model}

When the input to the Theta model, Eq.~\eqref{eq:theta_model}, is constant, i.e., $I(t)= a =const.$, 
we can obtain additional Koopman eigenfunctions and invariants as the model is exactly solvable.
In this case, the general solution is given by
\begin{align}
\theta(t)
  = 2\arctan\Bigl[\sqrt{a}\tan\bigl(\sqrt{a}(t+c)\bigr)\Bigr], 
\end{align}
where $c$ is a constant determined by the initial condition. By rearranging this expression, we obtain

\begin{align}
    \label{eq:eigenfunction_for_theta_model}
    C\exp(t) = \exp\left(\frac{1}{\sqrt{a}}\arctan\Bigl(
        \frac{1}{\sqrt{a}}\tan\frac{\theta}{2}
      \Bigr)\right),
\end{align}
where $C = \exp(c)$ is a constant.

From this form, it is evident that $\exp\left(\frac{1}{\sqrt{a}}\arctan\Bigl(\frac{1}{\sqrt{a}}\tan\frac{\theta}{2}\Bigr)\right)$ is a Koopman eigenfunction with eigenvalue $1$, since
\begin{align}
    \frac{d}{dt} \exp\left(\frac{1}{\sqrt{a}}
      \arctan\Bigl(
        \frac{1}{\sqrt{a}}
        \tan\frac{\theta}{2}
      \Bigr)\right) = \exp\left(\frac{1}{\sqrt{a}}
      \arctan\Bigl(
        \frac{1}{\sqrt{a}}
        \tan\frac{\theta}{2}
      \Bigr)\right).
\end{align}
Moreover, the ratio of this function, such as
\begin{align}
    \frac{ \exp\left(\frac{1}{\sqrt{a}}
      \arctan\Bigl(
        \frac{1}{\sqrt{a}}
        \tan\frac{\theta_j}{2}
      \Bigr)\right)}{ \exp\left(\frac{1}{\sqrt{a}}
      \arctan\Bigl(
        \frac{1}{\sqrt{a}}
        \tan\frac{\theta_k}{2}
      \Bigr)\right)},
\end{align}
gives invariants in addition to the $N-3$ invariants.


\begin{thebibliography}{99}

\bibitem{prigogine1968symmetry}
Prigogine I, Lefever R. 1968  Symmetry Breaking Instabilities in Dissipative Systems. {\em The Journal of Chemical Physics} \textbf{48}, 1695--1700.
(\href{http://dx.doi.org/10.1063/1.1668896}{10.1063/1.1668896})

\bibitem{Kuramoto1984chemical}
Kuramoto Y. 1984 {\em Chemical Oscillations, Waves, and Turbulence}.
New York: Springer.

\bibitem{winfree1967phase}
Winfree AT. 1967  Biological rhythms and the behavior of populations of coupled oscillators. {\em Journal of Theoretical Biology} \textbf{16}, 15--42.
(\href{http://dx.doi.org/10.1016/0022-5193(67)90051-3}{10.1016/0022-5193(67)90051-3})

\bibitem{fitzhugh1955threshold}
FitzHugh R. 1955  Mathematical models of threshold phenomena in the nerve membrane. {\em Bulletin of Mathematical Biophysics} \textbf{17}, 257--278.
(\href{http://dx.doi.org/10.1007/BF02477753}{10.1007/BF02477753})

\bibitem{nagumo1962pulse}
Nagumo J, Arimoto S, Yoshizawa S. 1962  An active pulse transmission line simulating nerve axon. {\em Proceedings of the IRE} \textbf{50}, 2061--2070.
(\href{http://dx.doi.org/10.1109/JRPROC.1962.288235}{10.1109/JRPROC.1962.288235})

\bibitem{ermentrout1986parabolic}
Ermentrout GB, Kopell N. 1986  Parabolic bursting in an excitable system coupled with a slow oscillation. {\em SIAM Journal on Applied Mathematics} \textbf{46}, 233--253.
(\href{http://dx.doi.org/10.1137/0146017}{10.1137/0146017})

\bibitem{Goldbeter1995circadian}
Goldbeter A. 1995  A model for circadian oscillations in the Drosophila period protein (PER). {\em Proceedings of the Royal Society B: Biological Sciences} \textbf{261}, 319--324.
(\href{http://dx.doi.org/10.1098/rspb.1995.0153}{10.1098/rspb.1995.0153})

\bibitem{kori2017jetlag}
Kori H, Yamaguchi Y, Okamura H. 2017  Accelerating recovery from jet lag: Prediction from a multi-oscillator model and its experimental confirmation in model animals. {\em Journal of Biological Rhythms} \textbf{32}, 44--58.
(\href{http://dx.doi.org/10.1177/0748730416671742}{10.1177/0748730416671742})

\bibitem{Watanabe1994constants}
Watanabe S, Strogatz SH. 1994  Constants of Motion for Superconducting Josephson Arrays. {\em Physica~D} \textbf{74}, 197--253.

\bibitem{sakaguchi1986soluble}
Sakaguchi H, Kuramoto Y. 1986  A soluble active rotater model showing phase transitions via mutual entertainment. {\em Progress of Theoretical Physics} \textbf{76}, 576--581.
(\href{http://dx.doi.org/10.1143/PTP.76.576}{10.1143/PTP.76.576})

\bibitem{daido1992order}
Daido H. 1992  Order function and macroscopic mutual entrainment in uniformly coupled limit-cycle oscillators. {\em Progress of theoretical physics} \textbf{88}, 1213--1218.

\bibitem{mauroy2013isostables}
Mauroy A, Mezi^^c4^^87 I, Moehlis J. 2013  Isostables, Isochrons, and Koopman Spectrum for the Action--Angle Representation of Dynamical Systems. {\em Physica D: Nonlinear Phenomena} \textbf{261}, 19--30.
(\href{http://dx.doi.org/10.1016/j.physd.2013.06.015}{10.1016/j.physd.2013.06.015})

\bibitem{nakao2016phase}
Nakao H. 2016  Phase reduction approach to synchronisation of nonlinear oscillators. {\em Contemporary physics} \textbf{57}, 188--214.

\bibitem{shirasaka2017phase}
Shirasaka S, Kurebayashi W, Nakao H. 2017  Phase--Amplitude Reduction of Transient Dynamics Far from Attractors for Limit-Cycling Systems. {\em Chaos} \textbf{27}, 023119.
(\href{http://dx.doi.org/10.1063/1.4977728}{10.1063/1.4977728})

\bibitem{yawata2024autoencoder}
Yawata K, Fukami K, Taira K, Nakao H. 2024  Phase autoencoder for limit-cycle oscillators. {\em Chaos: An Interdisciplinary Journal of Nonlinear Science} \textbf{34}, 063111.
(\href{http://dx.doi.org/10.1063/5.0205718}{10.1063/5.0205718})

\bibitem{namura2026arbitrary}
Namura N, Muolo R, Nakao H. 2026  Optimal interaction functions realizing higher-order Kuramoto dynamics with arbitrary limit-cycle oscillators. {\em Chaos: An Interdisciplinary Journal of Nonlinear Science} \textbf{36}, 023120.
(\href{http://dx.doi.org/10.1063/5.0307452}{10.1063/5.0307452})

\bibitem{ozawa2026delay}
Ozawa A, Kawamura Y. 2026  Phase reduction of reaction^^e2^^80^^93diffusion systems with delay. {\em Chaos: An Interdisciplinary Journal of Nonlinear Science} \textbf{36}, 033107.
(\href{http://dx.doi.org/10.1063/5.0313301}{10.1063/5.0313301})

\bibitem{Watanabe1993integrability}
Watanabe S, Strogatz SH. 1993  Integrability of a Globally Coupled Oscillator Array. {\em Physical Review Letters} \textbf{70}, 2391--2394.
(\href{http://dx.doi.org/10.1103/PhysRevLett.70.2391}{10.1103/PhysRevLett.70.2391})

\bibitem{Marvel2009identical}
Marvel SA, Mirollo RE, Strogatz SH. 2009  Identical Phase Oscillators with Global Sinusoidal Coupling Evolve by M{\"o}bius Group Action. {\em Chaos} \textbf{19}, 043104.
(\href{http://dx.doi.org/10.1063/1.3247089}{10.1063/1.3247089})

\bibitem{Pikovsky2008partially}
Pikovsky A, Rosenblum M. 2008  Partially Integrable Dynamics of Hierarchical Populations of Coupled Oscillators. {\em Physical Review Letters} \textbf{101}, 264103.
(\href{http://dx.doi.org/10.1103/PhysRevLett.101.264103}{10.1103/PhysRevLett.101.264103})

\bibitem{hong2011conformists}
Hong H, Strogatz SH. 2011  Conformists and contrarians in a Kuramoto model with identical natural frequencies. {\em Phys. Rev. E} \textbf{84}, 046202.
(\href{http://dx.doi.org/10.1103/PhysRevE.84.046202}{10.1103/PhysRevE.84.046202})

\bibitem{atsumi2012persistent}
Atsumi Y, Nakao H. 2012  Persistent fluctuations in synchronization rate in globally coupled oscillators with periodic external forcing. {\em Physical Review E―Statistical, Nonlinear, and Soft Matter Physics} \textbf{85}, 056207.

\bibitem{lohe2018vector}
Lohe MA. 2018  Higher-dimensional generalizations of the Watanabe^^e2^^80^^93Strogatz transform for vector models of synchronization. {\em Journal of Physics A: Mathematical and Theoretical} \textbf{51}, 225101.
(\href{http://dx.doi.org/10.1088/1751-8121/aac030}{10.1088/1751-8121/aac030})

\bibitem{park2022kinetic}
Park H. 2022  The Watanabe-Strogatz transform and constant of motion functionals for kinetic vector models. {\em Journal of Differential Equations} \textbf{323}, 113--151.
(\href{http://dx.doi.org/https://doi.org/10.1016/j.jde.2022.03.027}{https://doi.org/10.1016/j.jde.2022.03.027})

\bibitem{jain2025higher}
Jain JC, Jalan S. 2025  Low-dimensional Watanabe^^e2^^80^^93Strogatz approach for Kuramoto oscillators with higher-order interactions. {\em Chaos: An Interdisciplinary Journal of Nonlinear Science} \textbf{35}, 091104.
(\href{http://dx.doi.org/10.1063/5.0283600}{10.1063/5.0283600})

\bibitem{reid1972riccati}
Reid WT. 1972 {\em Riccati Differential Equations}.
New York: Academic Press.

\bibitem{cestnik2024integrability}
Cestnik R, Martens EA. 2024  Integrability of a Globally Coupled Complex Riccati Array: Quadratic Integrate-and-Fire Neurons, Phase Oscillators, and All in Between. {\em Physical Review Letters} \textbf{132}, 057201.
(\href{http://dx.doi.org/10.1103/PhysRevLett.132.057201}{10.1103/PhysRevLett.132.057201})

\bibitem{pazo2025spiking}
Paz\'o D, Cestnik R. 2025  Low-dimensional dynamics of globally coupled complex Riccati equations: Exact firing-rate equations for spiking neurons with clustered substructure. {\em Phys. Rev. E} \textbf{111}, L052201.
(\href{http://dx.doi.org/10.1103/PhysRevE.111.L052201}{10.1103/PhysRevE.111.L052201})

\bibitem{Koopman1931hamiltonian}
Koopman BO. 1931  Hamiltonian Systems and Transformation in Hilbert Space. {\em Proceedings of the National Academy of Sciences} \textbf{17}, 315--318.
(\href{http://dx.doi.org/10.1073/pnas.17.5.315}{10.1073/pnas.17.5.315})

\bibitem{Koopman1932spectra}
Koopman BO, {von Neumann} J. 1932  Dynamical Systems of Continuous Spectra. {\em Proceedings of the National Academy of Sciences} \textbf{18}, 255--263.
(\href{http://dx.doi.org/10.1073/pnas.18.3.255}{10.1073/pnas.18.3.255})

\bibitem{lasota2013chaos}
Lasota A, Mackey MC. 2013 {\em Chaos, fractals, and noise: stochastic aspects of dynamics} vol.~97.
Springer Science \& Business Media.

\bibitem{mezic2005spectral}
Mezi{\'c} I. 2005  Spectral Properties of Dynamical Systems, Model Reduction and Decompositions. {\em Nonlinear Dynamics} \textbf{41}, 309--325.
(\href{http://dx.doi.org/10.1007/s11071-005-2824-x}{10.1007/s11071-005-2824-x})

\bibitem{budivsic2012applied}
Budi{\v{s}}i{\'c} M, Mohr R, Mezi{\'c} I. 2012  Applied koopmanism. {\em Chaos: An Interdisciplinary Journal of Nonlinear Science} \textbf{22}.

\bibitem{mezic2013analysis}
Mezi{\'c} I. 2013  Analysis of fluid flows via spectral properties of the Koopman operator. {\em Annual review of fluid mechanics} \textbf{45}, 357--378.

\bibitem{mauroy2020koopman}
Mauroy A, Susuki Y, Mezi{\'c} I. 2020 {\em The {K}oopman Operator in Systems and Control}.
Springer.

\bibitem{brunton2022modern}
Brunton SL, Budi{\v{s}}i\'{c} M, Kaiser E, Kutz JN. 2022  Modern {K}oopman Theory for Dynamical Systems. {\em SIAM Review} \textbf{64}, 229--340.
(\href{http://dx.doi.org/10.1137/21M1425367}{10.1137/21M1425367})

\bibitem{Gaspard_1998}
Gaspard P. 1998 {\em Chaos, Scattering and Statistical Mechanics}.
Cambridge Nonlinear Science Series. Cambridge University Press.

\bibitem{taga2021koopman}
Taga K, Kato Y, Kawahara Y, Yamazaki Y, Nakao H. 2021  Koopman spectral analysis of elementary cellular automata. {\em Chaos: An Interdisciplinary Journal of Nonlinear Science} \textbf{31}.

\bibitem{parker2023koopman}
Parker JP, Valva C. 2023  Koopman analysis of the periodic Korteweg--de Vries equation. {\em Chaos: An Interdisciplinary Journal of Nonlinear Science} \textbf{33}.

\bibitem{taga2024dynamic}
Taga K, Kato Y, Yamazaki Y, Kawahara Y, Nakao H. 2024  Dynamic mode decomposition for Koopman spectral analysis of elementary cellular automata. {\em Chaos: An Interdisciplinary Journal of Nonlinear Science} \textbf{34}.

\bibitem{susuki2016power}
Susuki Y, Mezic I, Raak F, Hikihara T. 2016  Applied Koopman operator theory for power systems technology. {\em Nonlinear Theory and Its Applications, IEICE} \textbf{7}, 430--459.
(\href{http://dx.doi.org/10.1587/nolta.7.430}{10.1587/nolta.7.430})

\bibitem{hu2020synchronization}
Hu J, Lan Y. 2020  Koopman analysis in oscillator synchronization. {\em Phys. Rev. E} \textbf{102}, 062216.
(\href{http://dx.doi.org/10.1103/PhysRevE.102.062216}{10.1103/PhysRevE.102.062216})

\bibitem{wang2021probing}
Wang S, Lan Y. 2021  Probing the Phase Space of Coupled Oscillators with Koopman Analysis. {\em Physical Review E} \textbf{104}, 034211.
(\href{http://dx.doi.org/10.1103/PhysRevE.104.034211}{10.1103/PhysRevE.104.034211})

\bibitem{mihara2022basin}
Mihara A, Zaks M, Macau EEN, Medrano-T RO. 2022  Basin sizes depend on stable eigenvalues in the Kuramoto model. {\em Phys. Rev. E} \textbf{105}, L052202.
(\href{http://dx.doi.org/10.1103/PhysRevE.105.L052202}{10.1103/PhysRevE.105.L052202})

\bibitem{thibeault2025kuramoto}
Thibeault V, Claveau B, Allard A, Desrosiers P. 2025  Kuramoto meets Koopman: Constants of motion, symmetries, and network motifs. {\em arXiv preprint arXiv:2504.06248}.

\bibitem{kato2021asymptotic}
Kato Y, Zhu J, Kurebayashi W, Nakao H. 2021  Asymptotic phase and amplitude for classical and semiclassical stochastic oscillators via Koopman operator theory. {\em Mathematics} \textbf{9}, 2188.

\bibitem{Note1}
Note1.
We note that $\psi ^N_{\protect \bm {q}}\protect \notin L^2([0,2\pi ]^N)$ because it diverges when $\theta _{q^{(j)}}=\theta _{q^{(j+1)}}$, which can occur when $\theta _i = \theta _j$ for some $(i, j)$. In practice, we therefore consider a domain that excludes neighborhoods of these collision sets, for instance $X_\varepsilon = \left \{(\theta _1,\protect \ldots ,\theta _N)\ \middle |\ |\theta _i-\theta _j|>\varepsilon \ \protect \text {for all } i\protect \neq j \right \}$ with some $\varepsilon >0$. For identical oscillators evolving under a smooth vector field, the induced flow is unique and invertible for any finite time; hence phases cannot coincide in finite time unless they coincide initially. We use $X_\varepsilon $ whenever needed to keep $\psi _{\protect \bm {q}}^N$ well behaved.

\bibitem{Note2}
Note2.
We note that other invariants that do not take the form of the cross ratio can also exist; see the examples of the Theta model in Appendix~C.

\bibitem{ermentrout2008ekmodel}
Ermentrout B. 2008  Ermentrout-Kopell canonical model. {\em Scholarpedia} \textbf{3}, 1398.
(\href{http://dx.doi.org/10.4249/scholarpedia.1398}{10.4249/scholarpedia.1398})

\bibitem{tanaka2011multistable}
Tanaka T, Aoyagi T. 2011  Multistable attractors in a network of phase oscillators with three-body interactions. {\em Physical Review Letters} \textbf{106}, 224101.

\bibitem{skardal2020higher}
Skardal PS, Arenas A. 2020  Higher-order interactions in complex networks of phase oscillators promote abrupt synchronization switching. {\em Communications Physics} \textbf{3}, 218.
(\href{http://dx.doi.org/10.1038/s42005-020-00485-0}{10.1038/s42005-020-00485-0})

\bibitem{leon2024anomalous}
Le^^c3^^b3n I, Muolo R, Hata S, Nakao H. 2024  Higher-order interactions induce anomalous transitions to synchrony. {\em Chaos} \textbf{34}, 013105.
(\href{http://dx.doi.org/10.1063/5.0172585}{10.1063/5.0172585})

\bibitem{leon2025theory}
Le^^c3^^b3n I, Muolo R, Hata S, Nakao H. 2025  Theory of phase reduction from hypergraphs to simplicial complexes: A general route to higher-order Kuramoto models. {\em Physica D: Nonlinear Phenomena} p. 134858.
(\href{http://dx.doi.org/https://doi.org/10.1016/j.physd.2025.134858}{https://doi.org/10.1016/j.physd.2025.134858})

\bibitem{fujii2025emergence}
Fujii N, Taga K, Muolo R, Rink B, Nakao H. 2025  Emergence of higher-order interactions in systems of coupled Kuramoto oscillators with time delay. {\em arXiv preprint arXiv:2512.16193}.

\end{thebibliography}
\end{document}